\documentclass{class_for_arxiv}

\usepackage{natbib}
 \NatBibNumeric
 \def\bibfont{\small}%
 \bibpunct[, ]{[}{]}{,}{n}{}{,}%

\usepackage[colorlinks=true,breaklinks=true,bookmarks=true,urlcolor=blue,
     citecolor=blue,linkcolor=blue,bookmarksopen=false,draft=false]{hyperref}

\def\EMAIL#1{\href{mailto:#1}{#1}}

\TheoremsNumberedThrough     

\EquationsNumberedThrough    

\usepackage[english]{babel}

\usepackage{bold-extra}

\usepackage[utf8]{inputenc} 
\usepackage[T1]{fontenc}    
\usepackage{hyperref}       
\usepackage{url}            
\usepackage{booktabs}       
\usepackage{nicefrac}       
\usepackage{microtype}      
\usepackage{xcolor}         
\usepackage{dsfont,mathtools}
\usepackage[ruled,vlined]{algorithm2e}
\usepackage{subcaption}
\usepackage{graphicx}
\usepackage{optidef}
\usepackage{enumitem}

\graphicspath{{}{figure/}}

\usepackage[textwidth=\linewidth]{todonotes}		

\newcommand{\pp}{\mathbf{P}}
\newcommand{\rr}{\mathbf{R}}

\newcommand{\bm}{\mathbf{m}}
\newcommand{\mm}{\bm}

\newcommand{\calS}{\mathcal{S}}

\newcommand{\bY}{\mathbf{Y}}

\newcommand{\by}{\mathbf{y}}
\newcommand{\byone}{\mathbf{y}_{\cdot,1}}
\newcommand{\byopt}{\{\mathbf{y}^*(t)\}_{0 \le t \le T-1}}
\newcommand{\byoptt}{\mathbf{y}^{t*}}

\newcommand{\bmoptt}{\mathbf{m}^{t*}}

\newcommand{\yopt}{y^*}
\newcommand{\mopt}{m^*}
\newcommand{\yoptt}{y^{t*}}

\newcommand{\bv}{\mathbf{v}}

\providecommand{\norm}[1]{\| #1\|}

\newcommand{\R}{\mathbb{R}}

\newcommand\ball{\mathcal{B}}

\newcommand{\opt}{V^{(N)}_{\mathrm{opt}} (\bm(0),T)}
\newcommand{\rel}[1][N]{V^{}_{\mathrm{rel}}(\bm(0),T)}

\newcommand{\wa}{\emph{water}}

\newcommand\bS{\mathbf{S}}
\newcommand\bA{\mathbf{A}}
\newcommand\ba{\mathbf{a}}
\newcommand\bs{\mathbf{s}}
\newcommand\bP{\mathbf{P}}
\newcommand\bM{\mathbf{M}}
\newcommand\beps{\boldsymbol{\varepsilon}}
\newcommand\bzero{\mathbf{0}}
\newcommand\bMN{\mathbf{M}^{(N)}}
\newcommand\bYN{\mathbf{Y}^{(N)}}

\newcommand\bEN{\mathbf{E}^{(N)}}
\newcommand\MN{M^{(N)}}
\newcommand\bMNpi{\mathbf{M}^{\pi,(N)}}
\newcommand\YN{Y^{(N)}}
\newcommand\bYNpi{\mathbf{Y}^{\pi,(N)}}
\newcommand\YNpi{Y^{\pi,(N)}}

\newcommand\N{\mathbb{N}}
\newcommand\bR{\mathbf{R}}
\newcommand\calO{\mathcal{O}}
\newcommand\calE{\mathcal{E}}
\newcommand\calON{\mathcal{O}(1/\sqrt{N})}
\newcommand\calOe{e^{-\mathcal{O}(N)}}

\newcommand{\expect}[1]{\mathbb{E}\left[#1\right]}
\newcommand{\var}[1]{\mathrm{var}\left[#1\right]}
\newcommand{\proba}[1]{\mathbb{P}\left(#1\right)}
\newcommand{\expectpi}[1]{\mathbb{E}_{\pi}\left[#1\right]}

\newcommand{\Vpi}{V^{(N)}_{\mathrm{\pi}} (\bm(0),T)}
\newcommand\piprio{\pi^{\mathrm{priority}}}

\newcommand{\abs}[1]{\left|#1\right|}
\newcommand{\floor}[1]{\left\lfloor#1\right\rfloor}
\newcommand{\norme}[1]{\left\| #1 \right\|_1}

\newcommand\splus{\mathcal{S}^{+}}
\newcommand\szero{\mathcal{S}^{0}}
\newcommand\sminus{\mathcal{S}^{-}}
\newcommand\sempty{\mathcal{S}^{\emptyset}}
\newcommand\shat{\widehat{\mathcal{S}}}

\newcommand\MNshat{\widehat{\MN_s}}

\newcommand{\Update}{V^{(N)}_{\mathrm{LP-update}}(\bm(0),T)}
\newcommand{\updates}[2][T-t]{V^{(N)}_{\mathrm{LP-update}}(\bMN(#2),#1)}

\newcommand{\relm}[1]{V_{\mathrm{rel}}(#1)}

\theoremstyle{theorem} 			

\newtheorem{prop}[theorem]{Proposition}
\newtheorem{lem}[theorem]{Lemma}
\newtheorem{thm}[theorem]{Theorem}
\newtheorem{rem}[theorem]{Remark}

\newcommand{\bigzero}{\mbox{\normalfont\Large\bfseries (0)}}
\newcommand{\bigid}{\mbox{\normalfont\Large\bfseries Id}}
\newcommand{\bigz}{\mbox{\normalfont\Large\bfseries Z}}

\begin{document}

\RUNAUTHOR{N. GAST, B. GAUJAL, and C. YAN}

\RUNTITLE{LP-based policies for restless bandits}

\TITLE{LP-based policies for restless bandits: necessary and sufficient conditions for (exponentially fast) asymptotic optimality}

\ARTICLEAUTHORS{
\AUTHOR{Nicolas GAST}
\AFF{Univ. Grenoble Alpes, Inria, CNRS, Grenoble INP, LIG, 38000
Grenoble, France \EMAIL{nicolas.gast@inria.fr}}
\AUTHOR{Bruno GAUJAL}
\AFF{Univ. Grenoble Alpes, Inria, CNRS, Grenoble INP, LIG, 38000
Grenoble, France \EMAIL{bruno.gaujal@inria.fr}}
\AUTHOR{Chen YAN}
\AFF{STATIFY, Inria, 38334 Saint Ismier, France; Biostatistics and Spatial Processes, INRAE, 84914 Avignon, France \EMAIL{chen.yan@inria.fr}}
}

\ABSTRACT{
We provide a framework to analyse control policies for the restless Markovian bandit model, under both finite and infinite time horizon. We show that when the population of arms goes to infinity, the value of the optimal control policy converges to the solution of a linear program (LP). We provide necessary and sufficient conditions for a generic control policy to be: i) asymptotically optimal; ii) asymptotically optimal with square root convergence rate; iii) asymptotically optimal with exponential rate. We then construct the LP-index policy that is asymptotically optimal with square root convergence rate on all models, and with exponential rate if the model is non-degenerate in finite horizon, and satisfies a uniform global attractor property in infinite horizon. We next define the LP-update policy, which is essentially a repeated LP-index policy that solves a new linear program at each decision epoch. We conclude by providing numerical experiments to compare the efficiency of different LP-based policies.
}



\MSCCLASS{Primary: 90C40; secondary: 90C05, 90B99}
\ORMSCLASS{Primary: dynamic programming/optimal control: Markov: finite state;
  secondary: programming: linear; Probability: Markov processes}

\HISTORY{Received April 12, 2022; revised February 17, 2023, and August 30, 2023; accepted 30 October 2023}


\maketitle

\section{\bfseries\scshape{Introduction}} \label{sec:introduction-and-related-works}

In this paper we investigate the famous  Markovian restless bandit problem (termed RB for short) over a finite and an infinite horizon. In this problem, a decision maker faces a bandit with $N$ arms, where each arm can be seen as a Markov decision process with two actions: active and passive. At each decision epoch, the decision maker chooses which $\alpha N$ of these $N$ arms to activate, with the goal of maximizing the expected total reward over a finite (or infinite) time-horizon. All transition kernels and state-dependent rewards are assumed to be known. The arms produce rewards and evolve independently, but are coupled through the budget constraint on the number of arms that can be activated at each decision epoch. The word "restless" refers to the transition kernel under the passive action being not necessarily the identity matrix, hence generalizes the classical rested bandit model in \citet{Gittins79banditprocesses}.

This problem arises in various domains and has numerous applications (see \citet{zhang2021restless} and the references therein for examples). Solving the problem exactly has been shown to be PSPACE-hard in \citet{Papadimitriou99thecomplexity}. Consequently, there has been substantial interest in developing approximate solutions whose performance are provably close to optimal, and at the same time require computations that do not grow exponentially with the number of arms $N$. We shall focus on the asymptotic regime where the arm population $N$ grows and the activation budget at each epoch, $\alpha N$, is proportional to $N$. This regime was first studied in \citet{10.2307/2984953} and has been of longstanding theoretical and practical interest.

\subsection*{Literature review} \label{subsec:related-works}
The pioneering work on this problem appears in \citet{whittle-restless}, who proposed the famous Whittle's Index Policy (WIP) on infinite horizon problems, and conjectured that the policy is asymptotically optimal, meaning that the optimality gap (the difference between the performance of the optimal policy and of WIP) converges to zero when $N$ goes to infinity. This conjecture has been proven to be true in \citet{WeberWeiss1990}, under an additional uniform global attractor property (termed UGAP for short); but is false in general, as shown by the four states counter-example provided in the same paper \citet{WeberWeiss1990}. A later work in \citet{gast2023exponential} actually shows that the optimality gap converges to zero exponentially fast with $N$ in almost all cases, which provides a theoretical explanation to the empirical good performance of WIP. 

One potential drawback of WIP is that it requires the technical condition of indexability on the RB. Many works have been devoted to computing the indices or testing indexability, e.g. \citet{article,math8122226,gast2023testing}, which makes WIP easily computable for indexable problems. Yet, we can not apply this policy if the RB is non-indexable. To circumvent this weakness, another approach, based on solving linear programs, is proposed in \citet{Ve2016.6}, where a set of LP-priority policies is defined from the solution of a linear program, and is shown to be all asymptotically optimal (assuming again the UGAP), regardless of indexability. WIP is inside this set of LP-priority policies, if the RB is indexable. We show in this paper that the asymptotic optimality proven in \citet{Ve2016.6} occurs at exponential rate as well under mild additional assumptions.  

Studying the problem under infinite horizon is theoretically interesting, but all these asymptotic optimality results mentioned previously rely on the UGAP, which in most cases can only be verified numerically, and may very well not be satisfied on certain problems (\citet{gast2023exponential}). This motivates another research direction that considers the corresponding finite horizon model using the linear program approach.

To the best of our knowledge, this idea first appears in \citet{hu2017asymptotically}, that applies time-dependent Lagrange multipliers to define a LP-based index policy, and shows subsequently that it is asymptotically optimal (i.e. achieving an $o(1)$ optimality gap). Note that for finite-horizon problem, the asymptotically optimal policies are not necessarily priority policies. This problem is generalized to multiple actions (instead of the two actions active and passive) in \citet{ZayasCabn2017AnAO}, where the authors propose a policy that achieves an $\mathcal{O}(\log N /\sqrt{N})$ optimality gap, and the very recent work \citet{zhang2022near} show that an $\calON$ convergence rate can be obtained under the discounted criterion. The problem has also been studied under the case of non-statistically identical arms in \citet{Brown2020IndexPA}. In this work, which seems largely independent of the previously cited papers and use radical different proof methodology, the authors propose a policy that achieves an $\calON$ optimality gap under this more general setting. As suggested by \citet{Brown2020IndexPA}, the convergence appears to be faster than $\calON$ on certain problems. In \citet{zhang2021restless}, the authors show that this is indeed the case and prove that one can obtain $\mathcal{O}(1/N)$ rate for problems that are \emph{non-degenerate}. Actually, by refining their policy, we later show in this paper that this $\mathcal{O}(1/N)$ rate can be further improved to be $\calOe$.

In this work we propose an alternative policy, called  LP-update, that solves a LP at each time step  leveraging the knowledge of current state. This policy has also been considered in a more general case with  multiple actions as well as multiple budget constraints in a followup paper \citet{gast2022lp}. A similar idea motivated by the certainty equivalent control from dynamic programming is also used  in \citet{gallego1997multiproduct,jasin2012re,wu2015algorithms}. In \citet{jasin2012re} and \citet{wu2015algorithms} the authors propose policies that achieve $\calO(1/N)$ convergence rate, provided that their LP is non-degenerate. Note that the meaning of "degeneracy" as considered in \citet{jasin2012re} and \citet{wu2015algorithms} is the classical notion of degeneracy in standard LP theory on the simplex method (see e.g. Section 5.3 of \citet{matousek2006understanding}), and should not be confused with the "degeneracy" considered in the current paper. There is nevertheless a link between these two notions, discussed further in \citet{gast2022lp}.


\subsection*{Summary of contributions} \label{subsec:summary-of-contribution}

In this paper, we provide a generic framework to study the relationship between restless bandit problem and the LP relaxations introduced in \citet{hu2017asymptotically} for the finite horizon and in \citet{Ve2016.6} for the infinite horizon.  In the aforementioned papers, it is shown that the value of the stochastic control problem with $N$ arms converges to the solution of this LP as $N$ goes to infinity. We go further and make the following contributions:
\begin{enumerate} [label=\roman*)]
  \item The first contribution is to provide a new general framework to study the asymptotic performance of control policies for finite horizon RB. In this framework, any admissible policy is a \emph{deterministic} map from arm distribution vectors to decision vectors, which is independent to the arm population $N$. This dependence is only restored later by applying a randomized rounding technique, discussed in Section \ref{subsec:admissible-policy-and-randomized-rounding}.  The advantage of this approach is that it allows us to analyse the asymptotic optimality together with the convergence rate of any policy, by simply investigating properties of these deterministic maps. More precisely, we show that
       \begin{enumerate} [label=\alph*)]
         \item A {\it continuous} policy is asymptotically optimal if and only if it is LP-{\it compatible} (defined in Section \ref{subsec:LP-compatible}).
         \item If in addition the policy is {\it Lipschitz continuous}, then the asymptotic optimality occurs at rate $\calON$.
         \item If in addition the policy is {\it locally linear} around the  LP solution, then the asymptotic optimality occurs at rate $\calOe$.
       \end{enumerate}
       These properties show that the asymptotic performance of a control policy is intimately linked with the LP relaxation.
  \item We use the above characterization to provide sufficient conditions for the existence of LP-compatible policies, and to provide an effective construction of such policies. In particular:
  \begin{enumerate}[label=\alph*)]
    \item For any finite horizon RB, there always exists a LP-compatible Lipschitz continuous policy.
    \item We show that to ensure the local linearity around the optimal LP solution as in (i) (c), it is necessary and sufficient for the RB to be \emph{non-degenerate}, a condition already introduced in \citet{zhang2021restless} and defined in Section \ref{subsec:def-nondegenerate}. Moreover, we exhibit a degenerate example in Section \ref{subsec:necessary-condition-for-exponential} for which no policy converges to the LP relaxation bound faster than the classical $\calON$ rate.
  \end{enumerate}
  We also show that the non-degeneracy property is almost equivalent to a property that we call \emph{rankability}, and that implies the existence of an asymptotically optimal priority policy.
  \item The above results show that there exist many policies that are asymptotically optimal. Yet, for a finite number of arms $N$, not all will perform equality good. To select the best policy for small $N$, we consider two possibilities of improvement: (1) the LP-index policy already introduced in \cite{Brown2020IndexPA}, and (2) the LP-update policy inspired from the \emph{certainty equivalent control} in dynamic programming.
  \item Our last contribution is to analyse the convergence rate of LP-based policies for RB in the {\it infinite horizon} case. Under the UGAP, we prove that the policy introduced in \citet{Ve2016.6} has an exponential convergence rate if the RB is non-degenerate. Our proof uses similar techniques as in \citet{gast2023exponential}.
\end{enumerate}

Note that the new approach we propose in the current paper first defines a deterministic map using the LP solution, and the policy is constructed  from this deterministic map. The asymptotic optimality, as well as the convergence rate are then transformed into studying properties of this map. This point of view distinguishes our work from the existing papers on asymptotic heuristics of the RB model, as in \cite{hu2017asymptotically,ZayasCabn2017AnAO,Brown2020IndexPA,Ve2016.6,gast2023exponential,zhang2021restless,zhang2022near}. In these works, the authors start by constructing some specific policy, mostly from the LP relaxation, and prove the corresponding convergence rate afterwards. We provide a new look on this problem, by first capturing essential properties of the model, so that \emph{any} admissible policy satisfying such properties can achieve the desired convergence rate. The convergence results proven previously, e.g. the Lagrange index policy in \citet{Brown2020IndexPA}, the fluid-priority policies in \citet{zhang2021restless} can be seen as consequences of the more general results from our paper. By investigating the necessary and sufficient properties for guaranteeing a certain convergence rate, we aim at a better understanding of the fundamental complexity of the problem.

\subsection*{Outline}
The rest of the paper is organized as follows: Section \ref{sec:model-and-notations} defines the finite horizon RB model as well as the admissible policy. Section \ref{sec:a-hierarchy-of-policies} introduces a hierarchy of admissible policies, and prove asymptotic optimality (with convergence rate if possible) inside each of the hierarchy. Section \ref{sec:existenc-and-construction-of-policies} provides concrete constructions for the polices discussed in Section \ref{sec:a-hierarchy-of-policies}, and gives necessary and sufficient conditions for exponential convergence rate. Section \ref{sec:LP-update-policy} describes the LP indices and the LP-update policy. Section \ref{sec:infinite-horizon-case} deals with the infinite horizon case. Section \ref{sec:numerical-experiment} provides numerical studies and finally Section \ref{sec:conclusion-and-future-researches} concludes our work.

\section{\bfseries\scshape{Model description}} \label{sec:model-and-notations}

This paper is mainly focused on discrete time \emph{finite horizon restless bandit} (RB) models. The infinite horizon RB models will be considered in Section~\ref{sec:infinite-horizon-case}. We first describe the model in Section \ref{subsec:model-description}. We introduce the LP relaxation in Section \ref{subsec:LP-relaxation}. We define the admissible policy and the randomized rounding procedure in Section \ref{subsec:admissible-policy-and-randomized-rounding}, and we list our notational convention  in Section \ref{subsec:notaional-convention}.

\subsection{Finite horizon RB} \label{subsec:model-description}

A  finite horizon RB model is composed of $N$ statistically identical arms. Each arm can be considered as a Markov decision process (MDP) with a  finite state space $\mathcal{S} = \{1\dots d\}$. The state of the $n$th arm at the \emph{discrete} time $t \ge 0$ is denoted by $S_n(t)\in\{1\dots d\}$. The state of all the arms at time $t$ is denoted by $\bS(t) = \big( S_1(t), \dots, S_N(t) \big)$. At each time $t$, a decision maker observes $\bS(t)$ and chooses a fraction $0 < \alpha < 1$ of the $N$ arms to be activated. In the literature, some researchers study the problem under the non-binding constraint that \emph{at most} a fraction $\alpha$ of arms can be activated at each time (e.g. \citet{Brown2020IndexPA}, \citet{Ve2016.6}). By adding $\alpha N$ dummy arms that never change states and give zero rewards, we transform the non-binding setting into the binding setting since, for a given set of active arms, activating additional dummy arms does not modify the behavior of the system. Conversely, if we replace the active rewards $R^1_s$ by $R^1_s + R'$ with a large enough overall positive constant $R'$, we retrieve the non-binding setting from the binding one.

Note that in our model we do not need to assume $\alpha N$ to be an integer. If it is not, then a coin is tossed at the beginning of each decision epoch and the decision maker has to activate $\floor{\alpha N}+1$ arms with probability $\{\alpha N\} = \alpha N - \floor{\alpha N}$, and $\floor{\alpha N}$ arms with probability $1-\{\alpha N\}$, so that in expectation we are activating $\alpha N$ arms. This procedure will be justified later in Remark \ref{rem:comparison-with-other-paper}. We denote the action vector at time $t$ by $\bA(t) = \big( A_1(t),\dots, A_N(t) \big)$. For each arm that is in state $s$ and whose action is $a$, the decision maker earns an immediate reward $R^a_s\in\R$.

Given $S_n(t)=s$ and $A_n(t)=a$, the arm $n$ makes a Markovian transition to a state $s'$ with probability $P^a_{s,s'}$. Those transitions are independent among all arms: for given states $\bs,\bs'$ and activation vector $\ba$, one has:
\begin{align}
  \label{eq:markovian_evolution}
  \proba{\bS(t+1)=\bs' \mid \bS(t),\bA(t), \dots, \bS(0),\bA(0)} = \proba{\bS(t+1)=\bs' \mid \bS(t)=\bs, \bA(t)=\ba} = \prod_{n=1}^N P^{a_n}_{s_n, s'_{n}}.
\end{align}



By construction, the arms are exchangeable: two arms in the same state and for which the same action is chosen provide the same reward and have the same transition probabilities. This implies that the problem can be expressed by counting the number of arms in each state and the number of arms activated in each state. For a given state $s$, we denote by $\MN_s(t)$ the \emph{fraction} of arms in state $s$ at time $t$, and by $\YN_{s,a}(t)$ the \emph{fraction} of arms in state $s$ at time $t$ for which decision $a\in\{0, 1\}$ is taken.  We denote the corresponding vectors as $\bMN(t)\in\Delta^d$ and $\bYN(t) := \big(\YN_{s,1}(t),\YN_{s,0}(t)\big)_{s\in \{1\ldots d\}}\in\Delta^{2d}$, where $\Delta^d$ (and $\Delta^{2d}$) are the $d$-dimensional (and $2d$-dimensional) simplex of probability vectors.

We denote by $V^{(N)}_{\mathrm{opt}} (\bm(0),T)$ the maximal expected gain (per arm) that can be obtained by the decision maker:
\begin{maxi!}|s|{\bY \ge \mathbf{0}}{\mathbb{E} \Big[ \sum_{t=0}^{T-1} \sum_{s,a} R^{a}_{s} \YN_{s,a}(t) \Big]\label{eq1-1}}{\label{eq:prob1-reformulated}}{V^{(N)}_{\mathrm{opt}} (\bm(0),T) =}
  \addConstraint{\text{Arms follow the Markovian evolution \eqref{eq:markovian_evolution}}\label{eq2-trans}}{}
  \addConstraint{\YN_{s,0}(t)+\YN_{s,1}(t) = \MN_s(t)\ \ \forall t,s \label{eq2-2}}{}
  \addConstraint{\sum_{s} \YN_{s,1}(t) =
  \begin{cases}
  (\floor{\alpha N}+1)/N, & \mbox{with probability } \{\alpha N\} \\
  \floor{\alpha N}/N, & \mbox{otherwise}.
  \end{cases} \ \ \forall t \label{eq2-1}}{}
  \addConstraint{\MN_s(0) = m_s(0) \ \ \forall s, \label{eq:init1-1}}{}
\end{maxi!}
where $\bm(0) \in \Delta^d$ is the empirical measure of initial state vector: $m_s(0) = \frac1N\sum_{n=1}^N \mathbf{1}_{\{s_n(0)=s\}}$ for all $s\in\{1\dots d\}$. Note that \eqref{eq2-1} represent the constraints that $\alpha N$ of the $N$ arms must be activated at each time, and \eqref{eq:init1-1} correspond to the initial condition.


\subsection{LP relaxation} \label{subsec:LP-relaxation}
The key difficulty in the above optimization problem \eqref{eq:prob1-reformulated} is the constraint \eqref{eq2-1} that couples the evolution of all arms. The idea is to replace it by the relaxed constraint requiring that the \emph{expected} proportion of activated arms is $\alpha$ for all time steps $t$:
\begin{equation}\label{eq:relax-constraint}
  \sum_s \mathbb{E}_{\pi} \big[ \YN_{s,1}(t) \big] = \alpha, \ \ \forall t.
\end{equation}
The key property that makes this relaxed problem simpler is that it can then be rewritten entirely by using only the variables $y_{s,a} (t) := \expect{\YN_{s,a}(t)}$. To see that, we will show later in Lemma~\ref{lem:Markovian-transition-analysis} that the Markovian evolution \eqref{eq:Markovian-transition} implies that
\begin{align*}
  \expect{  \MN_s(t+1) \mid \bYN(t) = \by} = \sum_{s', a} y_{s',a}P^a_{s',s}.
\end{align*}
This implies that \eqref{eq2-trans} and \eqref{eq2-2} can be replaced by \eqref{eq:markov_relaxed} in the optimization problem below.  The rest of the costs and constraints then depend only on the expected number of arms in each state. We can therefore write the relaxed optimization problem as a linear problem with value $V_{\mathrm{rel}}(\bm(0),T)$:
\begin{maxi!}|s|{\by \ge \mathbf{0}}{\sum_{t=0}^{T-1} \sum_{s,a} R_s^a y_{s,a} (t)}{\label{eq:relaxed_problem}}{V_{\mathrm{rel}}(\bm(0),T)=}
  \addConstraint{y_{s,0}(t+1) + y_{s,1}(t+1)= \sum_{s',a} y_{s',a} (t)P^{a}_{s's}\label{eq:markov_relaxed}\qquad}{}{\forall s,t}
  \addConstraint{\sum_{s} y_{s,1} (t) = \alpha \label{eq:active}}{}{\forall t}
  \addConstraint{y_{s,0}(0) + y_{s,1}(0)=m_s(0)}{}{\forall s. \label{eq:init2}}
\end{maxi!}

In the above optimization problem, the constraints \eqref{eq:active} are the relaxation of the constraints \eqref{eq2-1}. They impose that the expected fraction of activated arms is $\alpha$ at all time. The constraints \eqref{eq:markov_relaxed} correspond to the expected behavior of the Markovian evolution of the system. Similarly, \eqref{eq:init2} correspond to the initial condition \eqref{eq:init1-1}.

Note that the optimization problem~\eqref{eq:relaxed_problem} does not depend on the arm population $N$. Moreover, as it is a relaxation of \eqref{eq:prob1-reformulated}, it should be clear that $\opt\le  \rel$. Since finding an optimal policy for $\opt$ is impractical, our strategy is to obtain information from optimal solutions to the linear program \eqref{eq:relaxed_problem} to construct policies whose values converge fast to $\rel$ as $N$ goes to infinity. As $\opt\le  \rel$, this will imply that they become asymptotically optimal as $N$ goes to infinity.

\subsection{Admissible policies and randomized rounding} \label{subsec:admissible-policy-and-randomized-rounding}

A policy determines which arms are made active at each decision epoch. In what follows, we focus on Markovian policies: such a policy is a sequence of decision rules $\pi=(\pi_0\dots \pi_{T-1})$ such that the decision rule $\pi_t:\Delta^d\to\Delta^{2d}$ specifies the fraction of arms for each action: if $\by=\pi_t(\bm)$, then when the empirical state vector at time $t$ is $\bm$, a fraction $y_{s,a}$ among the $m_s$ arms in state $s$ take action $a$. We say that a policy is \emph{admissible} if for all times $t$, all states $\bm\in\Delta^d$ and $\by=\pi_t(\bm)$, we have
\begin{align}
  \label{eq:admissible}
  y_{s,a} \geq 0, \quad \sum_{s} y_{s,1} = \alpha, \quad \text{ and } \sum_{a} y_{s,a} = m_s \quad \forall s,a.
\end{align}
We also say that a policy is continuous (respectively Lipschitz continuous) if for all $t$, $\pi_t$ is continuous (respectively Lipschitz continuous).


Note that our definition of admissible policies is independent of the arm population $N$. Moreover, an admissible policy does not assume that if $\by=\pi_t(\bm)$, then $Ny_{s,a}$ should be an integer. Hence, to make a policy applicable to the original problem with $N$ arms, we use a procedure that we call \emph{randomized rounding} that activates $N y_{s,1}$ arms in state $s$ \emph{in expectation} and that works as follows:
\begin{itemize}
  \item In a first pass, one activates $\floor{N y_{s,1}}$ arms in state $s$, and we let $z_s := N y_{s,1}-\floor{N y_{s,1}}$;
  \item In a second pass, one activates an extra $Z_s\in\{0,1\}$ arm in state $s$, such that for all $s$, $Z_s$ are random variables that satisfy $\expect{Z_s}=z_s\in[0,1)$, and $\sum_s Z_s=\sum_s z_s := h$ (almost surely).
\end{itemize}
Note that by definition, $h = \lfloor \alpha N \rfloor - \sum_s\floor{N y_{s,1}}$  or $h =  \lfloor \alpha N \rfloor +1 - \sum_s\floor{N y_{s,1}}$ and is therefore an integer. To do the second pass, one cannot simply generate the random variables $Z_s$ independently, because such  variables $Z_s$ may not  sum to exactly $h$.
An efficient algorithm to solve the above problem can be found in Section 5.2.3 of \citet{DBLP:journals/corr/IoannidisY16}. It has time complexity $\mathcal{O}(h d \cdot \log d)$.

\subsection{Notation and terminology convention} \label{subsec:notaional-convention}
Throughout our presentation, a bold letter (e.g. $\by$, $\bm$) denotes a vector whereas a normal letter (e.g. $y_{s,a}(t)$, $m_s(t)$) denotes a scalar. The bold letter $\bm$ always denotes a  state vector (that lives in $\Delta^d\subset\R^d$) whereas $\by = (\by_{.,1},\by_{.,0})$ denotes a state-action pair vector (that lives in $\Delta^{2d}\subset\R^{2d}$). For a vector $\bm\in\R^d$, we denote by $\norme{\bm}=\sum_{s} \abs{m_s}$ the $L_1$ norm of $\bm$, and $\ball(\bm^*,\varepsilon):=\{\bm \mid \norme{\bm-\bm^*}\le\varepsilon\}$ is the ball centered at $\bm^*$ of radius $\varepsilon$. Apart from rare cases, capital letters (e.g. $\bY$, $\bM$) denotes random variables whereas small letters denote deterministic values (e.g. $\by$, $\bm$). We write $\bYN$, $\bMN$ to emphasize the dependence on arm population $N$ so that each of its coordinate is of the form $k/N$ with $k \in \N$. The function $\mathbf{1}_{E}$ is a random variable that equals $1$ if the event $E$ occurs and $0$ otherwise. For a set $\mathcal{S}$, we use $\abs{\mathcal{S}}$ to denote its cardinal. By "asymptotic optimality", implicitly we are always comparing with the LP relaxation bound, which is an upper bound on the value of the optimal policy. We do so since the later is in general not easily computable.

\section{\bfseries\scshape{A hierarchy of policies}} \label{sec:a-hierarchy-of-policies}

In this section we introduce a hierarchy of admissible policies having increasingly desirable properties. We first give some preliminary results in Section~\ref{subsec:preliminary-result}. In Section \ref{subsec:LP-compatible}, we define the notion of LP-compatible policy and show that a continuous admissible policy is asymptotically optimal if and only if it is LP-compatible. If furthermore the policy is Lipschitz continuous, then we obtain a square root convergence rate. In Section~\ref{subsec:locally-linear-policy-and-exponential}, we show that if the policy is locally linear around one optimal LP solution, then the convergence rate can be improved to be exponential. Proofs of Lemma \ref{lem:Markovian-transition-analysis}, Theorem \ref{thm:CLT-rate} and Theorem \ref{thm:exponential-rate} are given respectively in Section \ref{subsubsec:proof1}, \ref{subsubsec:proof2} and \ref{subsubsec:proof3}.

\subsection{Evolution of $\MN(\cdot)$ for a given policy} \label{subsec:preliminary-result}

Assume that an admissible policy $\pi$ is given. To analyse the performance of such a policy, we will analyse how this policy makes the state evolve from $\MN(t)$ to $\MN(t+1)$. This evolution is decomposed in three steps: first the policy specifies $\bY(t)= \pi_t(\MN(t))$, which indicates the proportion of arms that should be activated \emph{on average}, then the randomized rounding procedure produces $\YN(t)$, which indicates how many arms should be activated. Lastly, a new state $\MN(t+1)$ is generated from $\YN(t)$.
This is summarized in the following diagram:
\begin{equation} \label{eq:map-chain}
  \bMN(t) \xrightarrow[\text{policy} \ \pi_t(\cdot)]{\text{admissible}} \bY(t) \xrightarrow[\text{rounding}] {\text{randomized}} \bYN(t) \xrightarrow[\text{Markovian transition~\eqref{eq:markovian_evolution}}]{\text{each arm follows the}} \bMN(t+1).
\end{equation}.
In this section, we analyse the Markovian transition that generates $\bMN(t+1)$ from $\bYN(t)$. To do so, we define the function $\phi:\Delta^{2d}\to\Delta^d$ that maps a vector $\by\in\Delta^{2d}$ to a vector $\phi(\by)=\big((\phi(\by))_1, \ \dots, \ (\phi(\by))_d \big)\in\Delta^d$ whose $s$th component is
\begin{equation} \label{eq:Markovian-transition}
  (\phi(\by))_s = \sum_{s', a} y_{s,a}P^a_{s',s}.
\end{equation}

The following lemma shows that $\bMN(t+1)$ is approximately equal to $\phi(\bYN(t))$ when $N$ is large (this is implied by \eqref{eq:CLT}), with an error that decreases as $\calON$. This observation will be used to show that a continuous admissible policy is optimal if and only if it is LP-compatible.  Equation~\eqref{eq:unbiased} shows that given $\bYN(t)$, $\bMN(t+1)$ is equal to $\phi(\bYN(t))$ on average. This fact, combined with the Hoeffding-type inequality \eqref{eq:hoeffding} and the fact that $\phi$ is linear, will be critically used in the proof of the exponential rate.
\begin{lem}  \label{lem:Markovian-transition-analysis} \
Let $\bEN(t) = \bMN(t+1) - \phi(\bYN(t))$, where $\phi(\cdot)$ is given in \eqref{eq:Markovian-transition}. We have:
  \begin{align}
    \expect{\bEN(t) \mid \bYN(t)} &= \mathbf{0},\label{eq:unbiased}\\
    \expect{\norme{\bEN(t)} \mid \bYN(t)} &\le \frac{\sqrt{d}}{\sqrt{N}},\label{eq:CLT}\\
    \proba{\norme{\bEN(t)} \ge \epsilon \mid \bYN(t)} &\le 2d e^{-2N \epsilon^2/d^2}.\label{eq:hoeffding}
  \end{align}
\end{lem}

A detailed proof of this result is provided in Section~\ref{ssec:proofs1}.

\subsection{LP-compatibility and  asymptotic  optimality} \label{subsec:LP-compatible}
For a given admissible policy $\pi$, we define  $\Vpi$ as the expected reward (per arm) when the system has $N$ arms and the policy $\pi$ is used. For a policy $\pi$, we also define $V_\pi(\bm(0),T) := \sum_{t=0}^{T-1}\sum_{a,s}R^a_sy^\pi_{s,a}(t)$, where $\by^\pi(t)$ is given by:
\begin{align*}
  \by^\pi(t)&=\pi_t(\bm^\pi(t))\\
  \bm^\pi(t+1)&=\phi(\by(t)).
\end{align*}
We say that a policy $\pi$ is \emph{LP-compatible} if there exists an optimal solution $\byopt$ of the LP \eqref{eq:relaxed_problem}, such that $\pi_t (\bm^*(t)) = \by^*(t)$ for all $0 \le t \le T-1$, where $\mopt_{s}(t)=\yopt_{s,0}(t)+\yopt_{s,1}(t)$. Following the above definition, an admissible policy is LP-compatible if and only if $V_\pi(\bm(0),T) = \rel$.

The following result makes the formal link between LP-compatible policy and asymptotically optimal policies for the $N$-arms bandit problem. In particular, it shows that a continuous policy $\pi$ is asymptotically optimal if and only if it is LP-compatible. In addition, the rate of convergence is $\calON$ when the policy is Lipschitz continuous. Note that this result alone provides necessary and sufficient conditions for asymptotically optimal policy, but does not guarantee the existence of such policies. We will show later in Section~\ref{sec:existenc-and-construction-of-policies} that for all finite horizon RB, there always exists a LP-compatible Lipschitz continuous policy that can be easily constructed.

\begin{thm}  \label{thm:CLT-rate} \
  Let $\pi = \left\{\pi_t \right\}_{0 \le t \le T-1}$ be an admissible and continuous policy. Then:
  \begin{align}
    \label{eq:thm:CLT1}
    \lim_{N\to\infty}\Vpi = V_\pi(\bm(0),T).
  \end{align}
  If in addition $\pi$ is Lipschitz continuous, then there exists a constant $C>0$ independent of $N$ such that
  \begin{align}
    \label{eq:thm:CLT2}
    \abs{\Vpi - V_\pi(\bm(0),T)} \le \frac{C}{\sqrt{N}}.
  \end{align}
  In particular, this implies that:
  \begin{enumerate}
    \item If $\pi$ is LP-compatible, then $\lim_{N\to\infty}\Vpi = \lim_{N\to\infty}\opt = \rel$.
    \item If $\pi$ is not LP compatible, then $\limsup_{N\to\infty}\Vpi < \rel$.
    \item If $\pi$ is LP-compatible and Lipschitz continuous, then there exists $\widetilde{C}>0$ independent of $N$ such that
    \begin{align*}
      \abs{\Vpi - \rel} \le \frac{\widetilde{C}}{\sqrt{N}}.
    \end{align*}
  \end{enumerate}
\end{thm}

\begin{proof}{Sketch of proof.} A detailed proof is presented in Section~\ref{ssec:proofs1}. We give here the main ideas.  Recall that $\Vpi=\expect{\sum_{t,a,s}R^a_s\YNpi_{s,a}(t)}$. By using the definition of $V_\pi(\bm(0),T)$ and the linearity of expectation, we have:
  \begin{align}
    \label{eq:th_conv}
    \Vpi - V_\pi(\bm(0),T) = \sum_{t,a,s}R^a_s\left(\expect{\YNpi_{s,a}(t)} - y^\pi_{s,a}(t)\right).
  \end{align}
  Consequently, showing that $\Vpi$ is close to $V_\pi$ is equivalent to showing that $\expect{\YNpi_{s,a}(t)}$ is close to $y^\pi_{s,a}$. In the detailed proof, we show it by recurrence on $t$ using two facts:
  \begin{itemize}
    \item The continuity of $\pi$ guarantees that if $\bm^\pi(t)$ and $\bMNpi(t)$ are close, then so are $\by^\pi(t)$ and $\bYNpi(t)$.
    \item Lemma~\ref{lem:Markovian-transition-analysis} shows that $\bMNpi(t+1) \approx \phi(\bYNpi(t))$, which implies that if $\by^\pi(t)$ and $\bYNpi(t)$ are close then so are $\bm^\pi(t+1)$ and $\bMNpi(t+1)$.\Halmos
  \end{itemize}
\end{proof}

\subsection{Locally linear policy and exponential convergence rate} \label{subsec:locally-linear-policy-and-exponential} As we have shown before, the LP-compatibility is a necessary and sufficient condition for a continuous policy to be asymptotically optimal. In this section, we show that when the policy is locally linear around an optimal solution, then this policy becomes optimal exponentially fast. Note that although LP-compatible policies always exist, this is not always the case for locally linear policies, as we shall see later in Section \ref{sec:existenc-and-construction-of-policies}.

We say that an LP-compatible policy $\pi = \left\{\pi_t \right\}_{0 \le t \le T-1}$ is \emph{locally linear} if there exists a solution $\byopt$ of \eqref{eq:relaxed_problem} such that for all $0 \le t \le T-1$, there exists $\varepsilon_t>0$ such that $\pi_t(\cdot)$ is \emph{linear} (more precisely, \emph{affine}) in the ball of radius $\varepsilon_t$ centered at $\bm^*(t)$, where $m^*_{s}(t) := y^*_{s,0}(t) + y^*_{s,1}(t)$ for all $s$. In other words, for each $t$, there exists a matrix $\mathbf{D} \in \mathbb{R}^{2d \times d}$ such that for all $\bm \in \ball(\bm^*(t),\varepsilon(t))$, we have $\pi_t(\bm) =  \mathbf{D} \cdot (\bm - \bm^*(t)) + \by^*(t)$.

\begin{thm} \label{thm:exponential-rate} \
  Consider a LP-compatible locally linear policy $\pi = \left\{\pi_t \right\}_{0 \le t \le T-1}$. There exists two constants $C_1,C_2>0$ independent of $N$ such that
  \begin{equation*}
    \abs{\Vpi - \rel} \le C_1 e^{-C_2 N}
  \end{equation*}
\end{thm}

We remark that the result of exponential convergence rate in Theorem \ref{thm:exponential-rate} is much stronger than the general square root rate given in Theorem \ref{thm:CLT-rate}. This is due to the locally linear condition. This local linearity around the optimal trajectory plays a key role in the proof of Theorem \ref{thm:exponential-rate}, as it is used in \eqref{eq:exchange-expectation} to justify the interchange of taking expectation with applying a linear function, in order to obtain \eqref{eq:exp-estimation1}. Our later discussion in Section \ref{subsec:necessary-condition-for-exponential} actually indicates that the local linearity is essentially necessary to obtain the exponential rate. A second key ingredient in the proof is the concentration inequality \eqref{eq:Hoeffding-estimation}, which relies on the fact that the $N$ arms are exchangeable. For the more general model where each arm of the bandit has its own state space (this has been considered in \citet{Brown2020IndexPA}), it is an interesting open question to see if we can formulate an exponential convergence type result in such generic case.

\subsection{Proof of results in Section \ref{sec:a-hierarchy-of-policies}}
\label{ssec:proofs1}

\subsubsection{Proof of Lemma~\ref{lem:Markovian-transition-analysis}} \label{subsubsec:proof1}

For simplicity of notation, let us denote by $\by := \bYN(t)$. There are $Ny_{s,a}$ arms in state $s$ and whose action is $a$ and each of these arms makes a transition to  state $s'$ with probability $P^a_{s,s'}$. This shows that $\MN(t+1)$ can be written as a sum of independent random variables as follows:
\begin{align*}
  \MN_{s'}(t+1) = \frac1N\sum_{s,a} \sum_{i=1}^{Ny_{s,a}} \mathbf{1}_{\{U_{s,a,i}\le P^a_{s,s'}\}},
\end{align*}
where the variables $U_{s,a,i}$ are i.i.d uniform random variable in $[0,1]$. Taking expectation then gives $\expect{\MN_{s'}(t+1) \mid \bYN(t)}=(\phi(\bYN(t)))_{s'}$, which gives \eqref{eq:unbiased}. It also implies that
\begin{align*}
  \expect{|E^{(N)}_{s'}(t+1)|^2 \mid \bYN(t)=\by} &= \var{\MN_{s'}(t+1)\mid \bYN(t)=\by}\\
  & = \frac1{N^2}\sum_{s,a}Ny_{s,a} P^a_{s,s'}(1-P^a_{s,s'}) \le \frac{\sum_{s,a}y_{s,a} P^a_{s,s'}}{N}.
\end{align*}
This shows that
\begin{align*}
  \expect{\norme{\bEN(t+1)} \mid \bYN(t)=\by} \le  \sqrt{d} \frac{ \sqrt{\sum_{s'} \sum_{s,a}y_{s,a}P^a_{s,s'}}}{\sqrt{N}} = \frac{\sqrt{d}}{\sqrt{N}},
\end{align*}
where the first inequality comes from Cauchy-Schwartz, and this gives \eqref{eq:CLT}.

Equation~\eqref{eq:hoeffding} is a direct consequence of Hoeffding's inequality. Indeed, one has
\begin{equation*}
  \proba{|E^{(N)}_s(t)|\ge\varepsilon/d \mid \bYN(t)} \le 2e^{-N\varepsilon^2/d^2}.
\end{equation*}
By using the union bound, this implies that
\begin{equation*}
  \proba{\norme{\bEN(t)}|\ge\varepsilon \mid \bYN(t)}\le d \cdot \proba{|E^{(N)}_s(t)|\ge\varepsilon/d \mid \bYN(t) } \le 2de^{-N\varepsilon^2/d^2}.
\end{equation*}

\subsubsection{Proof of Theorem~\ref{thm:CLT-rate}} \label{subsubsec:proof2}


  Let $\pi$ be a continuous policy.  We will first show by induction on $t$ that $\bMNpi(t)$ converges to $\bm^\pi(t)$ in probability as $N$ goes to infinity. This clearly holds for $t=0$ because $\bm^\pi(0)=\bMNpi(0)=\bm(0)$. Assume that this holds for some $t\ge0$, and let us show that this implies $\bYNpi(t)$ also converges to $\by^\pi(t)$ in probability. Indeed, we have
  \begin{align}
    \label{eq:th_conv.proof1}
    \norm{\by^\pi(t) - \bYNpi(t)} \le \norm{\pi_t(\bm^\pi(t)) - \pi_t(\bMNpi(t))} + \norm{\pi_t(\bMNpi(t)) - \bYNpi(t)}.
  \end{align}
  By construction of randomized rounding, $\norm{\pi_t(\bMNpi(t)) - \bYNpi(t)} \le d/N$. This shows that, by continuity of $\pi_t(\cdot)$, if $\bMNpi(t)$ converges in probability to $\bm^\pi(t)$, then $\bYNpi(t)$ also converges to $\by^\pi(t)$ in probability.

  For $\bMNpi(t+1)$ and $\bm^\pi(t+1)$, we have
  \begin{align}
    \label{eq:th_conv.proof2}
    \norm{\bm^\pi(t+1) - \bMNpi(t+1)} \le \norm{\phi(\by^\pi(t)) - \phi(\bYNpi(t))} + \norm{\bEN(t)}
  \end{align}
  As $\phi$ is continuous and $\bEN(t)$ converges to $\mathbf{0}$ in probability, this implies that $\bMNpi(t+1)$ converges to $\bm^\pi(t+1)$ in probability. This concludes the induction step. Consequently, $\bYNpi(t)$ converges in probability to $\by^\pi(t)$. As $\YNpi_{s,a}(t)\in[0,1]$ are bounded, the dominated convergence theorem implies that  $\lim_{N\to\infty}\expectpi{\YNpi_{s,a}(t)} = y^{\pi}_{s,a}(t)$, which by \eqref{eq:th_conv} implies \eqref{eq:thm:CLT1}.

  Assume now that for all $t$, $\pi_t$ is Lipschitz continuous. As $\phi$ is linear, $\phi$ is also Lipschitz continuous. Let $L$ be an upper bound on the Lipschitz constants of $\pi$ and $\phi$. Applying \eqref{eq:th_conv.proof2}, Lemma~\ref{lem:Markovian-transition-analysis} and \eqref{eq:th_conv.proof1}, we have:
  \begin{align*}
    \expect{\norm{\bm^\pi(t+1) - \bMNpi(t+1)}} &\le \expect{\norm{\phi(\by^\pi(t)) - \phi(\bYNpi(t))}} + \expect{\norm{\bEN(t)}}\\
    &\le L \expect{\norm{\by^\pi(t) - \bYNpi(t)}} + \sqrt{\frac{d}{N}}\\
    &\le L^2 \expect{\norm{\bm^\pi(t) - \bMNpi(t)}} + \frac{Ld}{N} + \sqrt{\frac{d}{N}}.
  \end{align*}
  By a direct induction on $t$ (which is essentially the discrete Gronwall's lemma), this implies that $\expect{\norm{\bm^\pi(t+1) - \bMNpi(t+1)}} = \calON$. Note however that the hidden constant in the $\mathcal{O}(\cdot)$ grows exponentially with time $t$. By \eqref{eq:th_conv}, this implies \eqref{eq:thm:CLT2}.

  To conclude the proof, one should note that a policy $\pi$ is LP-compatible if and only if $V_\pi(\bm(0),T) = \rel$. \Halmos

\subsubsection{Proof of Theorem~\ref{thm:exponential-rate}} \label{subsubsec:proof3}

  Let $\varepsilon:=\min_t \varepsilon_t$, and let $F_t:\Delta^d \rightarrow \Delta^{2d}$ be the linear function such that $\pi_t(\bm)=F_t(\bm)$ for $\bm\in\ball(\bm^*(t),\varepsilon)$. Denote by $\ell>0$ the Lipschitz constant of the linear map $\phi(\cdot)$, and by $L_t>0$ the Lipschitz constant of $F_t$ and write $L:= \max_t L_t$.

  Let $\delta:= \varepsilon/ (2 \left(1+\ell L+\dots+(\ell L)^T \right))$, and let us denote by $\calE(\delta)$ the event:
  \begin{equation*}
    \calE(\delta) :=  \left\{\text{for all $0 \le t \le T-1$: $\norm{\bEN(t)\le\delta}$}\right\},
  \end{equation*}
  where $\bEN(t)$ is defined as in Lemma~\ref{lem:Markovian-transition-analysis}, and let $\overline{\calE(\delta)} $ be the complementary of the event $\calE(\delta)$.

  By \eqref{eq:hoeffding} of Lemma \ref{lem:Markovian-transition-analysis}, we have
  \begin{equation}\label{eq:Hoeffding-estimation}
    \proba{\overline{\calE(\delta)}}  \le 2dT \cdot e^{-2N \delta^2/d^2}.
  \end{equation}

  Assume that event $\calE(\delta)$ holds. By definition of $\bEN(t)$ and \eqref{eq:map-chain}, we have
  \begin{align}
    \norme{\bMN(t+1) - \bm^*(t+1)}  &= \norme{\phi(\bYN(t)) + \bEN(t) - \phi(\pi_t(\bm^*(t)))} \nonumber \\
    \qquad & \le \norme{\phi(\bYN(t)) - \phi(\bY(t))} + \norme{\phi(\bY(t)) - \phi(\pi_t(\bm^*(t)))} + \norm{\bEN(t)} \nonumber \\
    \qquad & = \norme{\phi(\bYN(t)) - \phi(\bY(t))}  \nonumber \\
    \qquad &  \quad + \norme{\phi(\pi_t(\bMN(t))) - \phi(\pi_t(\bm^*(t)))} + \norm{\bEN(t)}  \nonumber \\
    \qquad & \le \frac{2d \ell}{N} + \ell L \cdot \norme{\bMN(t) - \bm^*(t)} + \delta. \label{eq:rounding-error-avoided}
  \end{align}
  A direct induction until $t=0$ then implies
  \begin{equation*}
    \norme{\bMN(t+1) - \bm^*(t+1)} \le \left( 1+\ell L+\dots+(\ell L)^t \right) \cdot (\delta + \frac{2d \ell}{N}).
  \end{equation*}
  This implies that $\bMN(t)$ is inside $\ball(\bm^*(t),\varepsilon)$ for all $0 \le t \le T-1$ and $N \ge 2d\ell / \delta$. As a side note, the term $2d \ell /N$ in \eqref{eq:rounding-error-avoided} and the assumption $N \ge 2d\ell / \delta$ will not appear, if the locally linear policy can be constructed as a time-dependent priority policy, as in Proposition~\ref{th:rankable} for rankable finite horizon RB, since then no randomized rounding is needed anywhere and $\bYN(t)=\bY(t)$ always holds.

  Consequently, we get:
  \begin{align}
  \expect{\bYN(t)\mathbf{1}_{\{\calE(\delta)\}}} -\by^*(t) \nonumber
   &= \expect{F_t(\bMN(t))\mathbf{1}_{\{\calE(\delta)\}}}-F_t(\bm^*(t)) \\ \nonumber
   &= \expect{F_t\left(\phi(\bYN(t-1)\mathbf{1}_{\{\calE(\delta)\}})\right)}-F_t\left(\phi(\by^*(t-1)\right)\\ \label{eq:exchange-expectation}
   &= F_t\circ\phi\left(\expect{\bYN(t-1)\mathbf{1}_{\{\calE(\delta)\}}}-\by^*(t-1)\right),
  \end{align}
  where on the last equality \eqref{eq:exchange-expectation} we have interchanged the expectation $\expectpi{\cdot}$ with $F_t \circ \phi (\cdot)$, which is possible since the later is a linear map. A direct induction on $t$ then implies that
\begin{align} \label{eq:exp-estimation1}
    \norme{\expect{\bYN(t)\mathbf{1}_{\{\calE(\delta)\}}} -\by^*(t)} \nonumber
   &\le L' \norme{\expect{\bYN(t-1)\mathbf{1}_{\{\calE(\delta)\}}} -\by^*(t-1)}\\
   &\le (L')^T \norme{\expect{\bYN(0)\mathbf{1}_{\{\calE(\delta)\}}} -\by^*(0)}.
\end{align}
where $L'$ is an upper bound on the Lipschitz constants of maps $F_t\circ\phi(\cdot)$ for $0 \le t \le T-1$. Moreover by \eqref{eq:Hoeffding-estimation}, we have
\begin{equation} \label{eq:exp-estimation2}
  \norme{\expect{\bYN(t)} - \expect{\bYN(t)\mathbf{1}_{\{\calE(\delta)\}}}} \le 2d \cdot \proba{\bar{\calE}(\delta)}\le 4d^2 T e^{-C_2 N},
\end{equation}
where $C_2 := 2\varepsilon^2/((1+\dots+L^{T-1})^2 d^2)$. Combining \eqref{eq:exp-estimation1} and \eqref{eq:exp-estimation2} gives
\begin{equation*}
    \norme{ \expect{\bYN(t)} - \by^*(t)} \le C_1 e^{-C_2 N},
  \end{equation*}
where we may choose $C_1 := 4 d^2 T^2 (1+(L')^T)$.  Consequently, by \eqref{eq:th_conv}, all locally linear LP-compatible policies are asymptotically optimal with exponential rate, and this concludes our proof.  \Halmos


\section{\bfseries\scshape{Existence and construction of policies}} \label{sec:existenc-and-construction-of-policies}

In this section we provide constructions of Lipschitz continuous policies and locally linear policies, defined in the previous Section \ref{sec:a-hierarchy-of-policies}. In Section \ref{subsec:def-nondegenerate} we define the non-degenerate and rankable RB. In Section \ref{subsec:water-filling}, we introduce the idea of "water-filling", and show that the policies induced by "water-filling" are LP-compatible Lipschitz continuous policies, and are furthermore locally linear policies if the RB is non-degenerate. We compare the non-degenerate condition with the rankable condition in Section \ref{subsubsec:rankable-vs-non-degenerate}. In Section \ref{subsec:necessary-condition-for-exponential}, we construct a degenerate $2$-dimensional RB over which no policy converges asymptotically fast to the LP solution. This implies that non-degeneracy is a necessary condition for the exponential convergence rate in general. Proofs of Theorem \ref{th:rankable} and Lemma \ref{lem:water-filling} are given respectively in Section \ref{subsubsec:proof4} and \ref{subsubsec:proof5}.

\subsection{Non-degenerate and rankable RB} \label{subsec:def-nondegenerate}

Let $\byopt$ be an optimal solution of the LP relaxed problem~\eqref{eq:relaxed_problem}. For each time $t$, we partition the set $\calS$ into four sets $\splus(t)$, $\szero(t)$, $\sminus(t)$ and $\sempty(t)$ as follows:
\begin{align*}
  \splus(t) &:= \{s \in \mathcal{S} \mid y^*_{s,1}(t) > 0 \ \text{and} \ y^*_{s,0}(t) = 0 \};\\
  \szero(t) &:= \{s \in \mathcal{S} \mid y^*_{s,1}(t) > 0 \ \text{and} \ y^*_{s,0}(t) > 0 \};\\
  \sminus(t)&:= \{s \in \mathcal{S} \mid y^*_{s,1}(t) = 0 \ \text{and} \ y^*_{s,0}(t) > 0 \};\\
  \sempty(t)&:= \{s \in \mathcal{S} \mid y^*_{s,1}(t) = 0 \ \text{and} \ y^*_{s,0}(t) = 0 \}.
\end{align*}
The intuition behind this partition is as follows: For the optimal relaxed solution $\by^*$, at time $t$, it is optimal to activate all arms whose state is in $\splus(t)$, a fraction of those whose state is in $\szero(t)$, and none of those whose state is in $\sminus(t)$. Also note that the optimal solution is such that at time $t$, there are no arms whose state is in $\sempty(t)$: for all $s\in\sempty(t)$, we have $m^*_{s}(t)=y^*_{s,0}(t)+y^*_{s,1}(t)=0$.

Following this intuitive definition, we construct below a LP-compatible Lipschitz continuous  policy that activates in priority the arms in set $\splus(t)$, then the ones in $\szero(t)$ and then the ones in $\sminus(t)$. As we shall see below, one has to be careful on how to deal with the arms in $\szero(t)$.

Before defining the water-filling policy, and for reasons that will become clear in Theorem~\ref{th:rankable} and Theorem~\ref{th:water-filling}, we introduce two definitions:
\begin{enumerate}
  \item A RB is \emph{rankable} if there exists an optimal solution of \eqref{eq:relaxed_problem} for which $\abs{\szero(t)}\le 1$ for all $t$. Otherwise we call this RB \emph{non-rankable}.
  \item A RB is \emph{non-degenerate} if there exists an optimal solution $\byopt$ of \eqref{eq:relaxed_problem} for which $\abs{\szero(t)}\ge 1$ for all $t$. Otherwise we call this RB \emph{degenerate}. This definition coincides with the one in \citet{zhang2021restless}.
\end{enumerate}

At first glance it appears that rankable and non-degenerate RB's are complementary to each other. Surprisingly, it turns out that in practice these two conditions are \emph{almost} equivalent when the solution of the LP is unique, as indicated by the next proposition. Moreover, this proposition also shows that when there are multiple solutions, testing if a problem is non-degenerate can be done by computing an interior solution. These two claims are stated formally in the next result, that we prove and comment in Section~\ref{subsubsec:rankable-vs-non-degenerate}.

\begin{prop}
  \label{th:rankable_or_degenerated} \
  \begin{enumerate}
    \item Consider a RB for which the LP problem~\eqref{eq:relaxed_problem} has a unique solution. If this RB is non-rankable, then it is degenerate.
    \item Consider a RB for which the LP problem~\eqref{eq:relaxed_problem} has multiple solutions. Denote by $\Theta$ the optimal face of the LP (i.e. the set of all optimal solutions). Then the RB is non-degenerate if and only if there exists an optimal solution $\byopt$ in the relative interior of $\Theta$ for which (by using this solution) $\abs{\szero(t)}\ge 1$ for all $t$. 
  \end{enumerate}

\end{prop}
We say that a policy is a \emph{(time-dependent) priority policy} if for all time $t$, there exists a permutation $\sigma=\sigma_1\dots\sigma_d$ of the states (that depends on $t$) such that the policy activates first the arms in state $\sigma_1$, then the ones in state $\sigma_2$, etc. up to activating a fraction $\alpha$ of arms. In other words, if the arm configuration vector at time $t$ is $\bm \in \Delta^d$, then the policy will activate $y_{s,1}$ arms in state $s$, where for all $i \in \{1\dots d\}$, $y_{\sigma_i,1}$ is defined as:
\begin{align}
  \label{eq:piecewise_affine}
  y_{\sigma_i, 1} := \pi^{\mathrm{priority}(\sigma)}_{\sigma_i,1}(\bm) = \min( m_{\sigma_i}, \alpha - \sum_{j=1}^{i-1}y_{\sigma_j, 1}).
\end{align}
The next theorem justifies the notion of rankable RB.
\begin{thm}\label{th:rankable} \
  A RB is rankable if and only if there exists a time-dependent priority policy that is asymptotically optimal.
\end{thm}
The proof of this result is postponed to Section~\ref{subsubsec:proof4}. As we shall see later, one can use any order inside $\splus(t)$ or $\sminus(t)$ and still obtain an asymptotically optimal policy (although some orders are better than others as we elaborate in Section~\ref{subsec:LP-indices} and Section~\ref{subsec:tie-solving-check}). Theorem~\ref{th:rankable} shows that one has to be careful on dealing with the states in $\szero(t)$: if the RB is non-rankable, i.e. if $\abs{\szero(t)} > 1$ for some $t$, one cannot simply use a fixed priority order between those states at time $t$ to obtain an asymptotically optimal policy. To do so, we shall introduce the idea of "water-filling".

\subsection{The water-filling policy} \label{subsec:water-filling}

At time $t$, the water-filling policy observes $\bMN(t) \in \Delta^d$ and decides $\bY(t) \in \Delta^{2d}$, where $Y_{s,1}(t)$ is the expected fraction of arms that are in state $s$ and should be activated (recall that $\bYN(t)$ is then generated from $\bY(t)$ by applying randomized rounding). This policy works as follows.  For ease of notation, we drop momentarily the $t$ from the notations and we assume that the states are ordered so that the first $\abs{\splus}$ states are in $\splus$, the next $\abs{\szero}$ states are in $\szero$, the next $\abs{\sminus}$ states are in $\sminus$, and finally the rest are in $\sempty$. We view the states as $d$ buckets enumerated from $1$ to $d$, where bucket number $1 \le s \le d$ has capacity $\MN_s$ and $\alpha$ is the total quantity of water that needs to be poured into these buckets. For each time $t$, we fill the buckets following the rules described below and in Algorithm \ref{algo:water-filling}:
\begin{enumerate}
  \item We first activate all arms in $\splus$ by using a strict priority order on the states $1, \cdots, \abs{\splus}$. The only constraint is to activate no more than what we have, i.e. $Y_{s,1} \le \MN_s$ for $s\in\splus$ (the first "for" loop in Algorithm \ref{algo:water-filling});
  \item If there is still some water left, we then activate states in $\szero$ by using a {\it reversed} priority order on the states, namely  $\abs{\splus}+\abs{\szero}, \cdots, \abs{\splus}+1$ with the constraint that $Y_{s,1} \le \min(\MN_s,y^*_{s,1})$ for $s \in \szero$ (the second "for" loop in Algorithm \ref{algo:water-filling});
  \item If there is still some water left, we then complete by activating states in $\szero$ and then in $\sminus$ and then in $\sempty$ by using the priority order $\abs{\splus}+1, \cdots, d$ (the third "for" loop in Algorithm \ref{algo:water-filling}).
\end{enumerate}
Note that if for all $t$ we have $\abs{\szero(t)} \le 1$, the water-filling policy becomes a time-dependent priority policy. The next lemma shows that the water-filling policy is LP-compatible Lipschitz continuous, and is furthermore locally linear if the RB is non-degenerate. As will be explained later, the reversed priority order in the second step above is essential for this local linearity.

\begin{algorithm}[hbtp]
\SetAlgoLined
\SetKwInput{KwInput}{Input}
\KwInput{Time horizon $T$ and initial configuration vector $\bMN(0) = \bm(0)$.}
  Solve the linear program \eqref{eq:relaxed_problem} with time horizon $T$ and initial configuration vector $\bm(0)$, obtain an optimal solution $\by^*$ \;
  Order the states so that the first $\abs{\splus}$ states are in $\splus$, the next $\abs{\szero}$ states are in $\szero$, the next $\abs{\sminus}$ states are in $\sminus$, and the rest are in $\sempty$ \;
 \For{$t = 0,1,2,\dots,T-1$}{
     Initialize $\wa := \alpha$ and $Y_{s,1}(t):=0$ for $s \in \calS$ \;
     \For{$s = 1,2,\dots,\abs{\splus}$}{$Y_{s,1}(t):=\min(\MN_s(t),\wa)$ \;
                                        $\wa \mathrel{-}= Y_{s,1}(t)$ \;
       }
     \For{$s = \abs{\splus}+\abs{\szero}$, $\abs{\splus}+\abs{\szero}-1$, \dots, $\abs{\splus}+1$}{$Y_{s,1}(t):=\min(\MN_s(t), \wa, y^*_{s,1}(t))$ \;
                                                                                                   $\wa \mathrel{-}= Y_{s,1}(t)$ \;
       }
     \For{$s = \abs{\splus}+1$, $\abs{\splus}+2$, \dots, $d$}{$\emph{second-pass} := \min(\MN_s(t)-Y_{s,1}(t),\wa)$ \;
                                                              $Y_{s,1}(t) \mathrel{+}= \emph{second-pass}$ \;
                                                              $\wa \mathrel{-}=  \emph{second-pass}$  \;
     }
   Set $Y_{s,0}(t) := \MN_s(t) - Y_{s,1}(t)$ for $s \in \calS$ \;
   Obtain $\bYN(t)$ from $\bY(t)$ as in \eqref{eq:map-chain}, apply randomized rounding described in Section \ref{subsec:admissible-policy-and-randomized-rounding} if necessary \;
   Use actions $\YN_{s,a}(t)$ over all arms to advance to the next time-step\;
  }
 \caption{The water-filling policy.}
 \label{algo:water-filling}
\end{algorithm}


\begin{lem} \label{lem:water-filling} \
  For any finite horizon RB, the water-filling policy described above is a LP-compatible Lipschitz continuous  policy. Moreover, if the RB is non-degenerate, i.e. if for all $t$, $\abs{\szero(t)}\ge1$, then the water-filling policy is a LP-compatible locally linear  policy. And if the RB is degenerate, then there is no LP-compatible locally linear  policy.
\end{lem}
The proof of Lemma~\ref{lem:water-filling} is postponed to Section~\ref{subsubsec:proof5}. A direct consequence of this lemma, combined with Theorem~\ref{thm:CLT-rate} and Theorem~\ref{thm:exponential-rate} is that the water-filling policy is asymptotically optimal at rate at least $\calON$.
\begin{thm}  \label{th:water-filling} \
  Consider a finite horizon RB. For any water-filling policy $\pi$ described in Algorithm \ref{algo:water-filling}, there exists $C>0$ independent of $N$ such that
  \begin{align}
    \label{eq:Gronwall-type-estimation}
    \abs{\Vpi - \rel}\le \frac{C}{\sqrt{N}}.
  \end{align}
  Moreover, if the RB is non-degenerate, then there exists $C_1,C_2>0$ such that:
  \begin{align}
    \label{eq:th:water-expo}
    \abs{\Vpi - \rel}\le C_1e^{-C_2N}.
  \end{align}
\end{thm}

Lemma~\ref{lem:water-filling} shows that the non-degenerate condition is necessary and sufficient for the existence of a LP-compatible locally linear  policy. Theorem~\ref{th:water-filling} is less precise in the sense that we only show that non-degeneracy is sufficient to obtain an exponentially asymptotically optimal policy. In Section~\ref{subsec:necessary-condition-for-exponential}, we provide an example of a RB that is degenerate and for which there are no exponentially fast asymptotically optimal policy, with respect to the LP relaxed bound. Although we do not prove it, we conjecture that this holds in general so that the non-degeneracy is also a necessary condition for \eqref{eq:th:water-expo} to hold.

\begin{rem}[Comparison with \citet{zhang2021restless}]  \label{rem:comparison-with-other-paper} \
  Note that the authors of \citet{zhang2021restless} introduce a class of fluid-priority policies (in their Algorithm~1) that is close to our definition of water-filling policy. However, there are  major differences between their algorithm and ours:
  
  \begin{itemize}
    \item When $Ny^*_{s,1}(t)$ is not an integer, the authors choose to round fractional number of arms into integer numbers in the water-filling procedure, e.g. no more than $ \lfloor N y^*_{s,1} \rfloor$ arms can be activated in state $s \in \szero$, whereas we consider the water-filling procedure as a map from any vector $\bm \in \Delta^d$ into the decision vector $\by \in \Delta^{2d}$, and apply a non-trivial randomized rounding technique afterwards to avoid these rounding errors, so that our randomized water-filling policy activates exactly $N y^*_{s,1}$ arms in expectation, while respecting the budget constraints. Note that in order to achieve an exponential convergence rate, it is essential to remove all possible rounding errors. This also explains why in our problem formulation in \eqref{eq2-1}, we need to apply a randomized procedure to activate exactly $\alpha N$ arms in expectation when $\alpha N$ is not an integer. Instead in \citet{zhang2021restless}, the authors choose to activate $\lfloor \alpha N \rfloor$ arms for non-integer values, this is another source of rounding error.

    \item When one needs to activate more than $Ny^*_{s,1}(t)$ arms in a state $s\in\szero(t)$, we do a second pass of water-filling algorithm by using a {reversed} order on $\szero(t)$ compared to the order used in first pass, whereas in Algorithm~1 of \citet{zhang2021restless} the orders of the two passes are not specified, and simply given by two "for" loops as in line $14$ and line $19$ of their Algorithm 1. This reversed order is essential to establish the local linearity of $\pi$ around $\bm^*$, which then allows us to use Theorem~\ref{thm:exponential-rate}.  
      \item Finally, in \citet{zhang2021restless} the authors only obtain the $\mathcal{O}(1/N)$ convergence rate for their algorithm on non-degenerate problems, by using a very different proof method.
    \end{itemize}

\end{rem}

\subsection{Proof of results in Section \ref{sec:existenc-and-construction-of-policies}}
\label{ssec:proofs_policies}

\subsubsection{Proof of Proposition~\ref{th:rankable_or_degenerated}}
\label{subsubsec:rankable-vs-non-degenerate}

For the first item, we can actually assert a slightly broader result. It posits that for any RB, the corresponding optimization problem \eqref{eq:relaxed_problem} has an optimal solution, $\byopt$, that adheres to
\begin{equation}\label{eq:total-randomization}
  \sum_{t=0}^{T-1} \abs{\szero(t)} \le T.
\end{equation}
To arrive at this conclusion, we can draw parallels with our formulation of the optimization problem as an MDP, as will be elaborated upon in Equation~\eqref{eq:prob3}. By reconfiguring the optimization problem \eqref{eq:relaxed_problem}, we can depict it as a $T$ time-steps MDP with $T$ constraints stemming from \eqref{eq:active}. A feasible (not necessarily optimal) solution $\by$ engenders a decision rule $w_{\by}$ for the MDP. This is characterized by selecting action $a$ in state $s$ at time-step $t$ with probability $w_{\by}(s,a,t) := \frac{y_{s,a}(t)}{\sum_{a'} y_{s,a'}(t)}$. If $\sum_{a'} y_{s,a'}(t) = 0$, the probability defaults to $w_{\by}(s,0,t) = 0$ and $w_{\by}(s,1,t) = 1$. The number of randomizations of the decision rule $w_{\by}$ is defined as the number of pairs $(s,t)$ such that $0<w_{\by}(s,a,t)<1$ for some action $a$. This means that the number of randomizations is $\sum_t|\szero(t)|$. As an implication of Theorem 3.8 in \citet{altman-constrainedMDP}, there exists an optimal solution $\by^*$ to \eqref{eq:relaxed_problem} such that number of randomizations of $w_{\by^*}$s is bounded by $T$.  The result in \eqref{eq:total-randomization} is naturally inferred.

The claim of the first item is then a direct consequence of Equation~\eqref{eq:total-randomization}: if there exists a unique solution, it must satisfy \eqref{eq:total-randomization}, which by the pigeonhole principle implies that either $\abs{\szero(t)}\le1$ for all $t$ (the problem is rankable) or there exists $t$ such that $\abs{\szero(t)}=0$ (the problem is degenerate).

For the second item, it is a consequence of the following fact: denote by $\by^{*}_{(1)}, \by^{*}_{(2)} \in \Theta$ two optimal solutions of the LP, then for any $0 < \lambda < 1$, $\by^{*}_{\lambda}:= \lambda \by^{*}_{(1)} + (1-\lambda) \by^{*}_{(2)}$ is also an optimal solution. For any $0 \le t \le T-1$, denote by $\szero_{(1)}(t), \szero_{(2)}(t)$ and $\szero_{\lambda}(t)$ the sets with respect to the optimal solutions $\by^{*}_{(1)}, \by^{*}_{(2)}$ and $\by^{*}_{\lambda}$ respectively, then $\szero_{(1)}(t) \cup \szero_{(2)}(t) \subset \szero_{\lambda}(t)$. The proof of this claim is straightforward from the definitions.    \Halmos

The first claim of Proposition \ref {th:rankable_or_degenerated} implies that under the assumption of a unique solution, a problem that is non-rankable cannot be non-degenerate.  This leaves the question as whether there is a problem that is both rankable and degenerate? The answer is yes and we provide a small example below.

\begin{example}[A rankable and degenerate problem]
  Let us consider a two states RB with a proportion of activation $\alpha=0.5$. The initial condition is $\mm(0)=[0.5,0.5]$, the rewards are $\bR^0=[0,0]$ and $\bR^1=[1,0]$, and the matrices are identity matrices: $\bP^0=\bP^1=\mathbf{I}$. The solution to the LP is clearly unique and consists of activating all arms in state $1$ and no arms in state $2$. Hence, $\abs{\szero(t)}=0$ for all $t$. This example is rankable and is also degenerate. \Halmos
\end{example}

As for computational issues concerning the second claim of Proposition \ref{th:rankable_or_degenerated}, it is useful to point out that there exist efficient methods to compute  an optimal solution in the relative interior of $\Theta$. While the simplex method always returns an extremal solution and will not work here,  there exist efficient algorithms in the literature to achieve this goal. In particular, in \citet{mehrotra1993finding} the authors propose a method  to find an interior point in the optimal face of a LP, whose complexity in the worst case is the same as the complexity of solving the LP. Combining this result with Proposition ~\ref{th:rankable_or_degenerated}, we conclude that testing non-degeneracy for RB admitting multiple optimal LP solutions is not harder than the case of having a unique optimal solution.

\subsubsection{Necessary condition for exponential convergence rate} \label{subsec:necessary-condition-for-exponential}

Consider a two states RB with horizon $T=2$ and proportion of activation $\alpha=0.5$. The initial condition is $\mm(0) = [0.5,0.5]$. The rewards are $\rr^0 = [0,0]$, $\rr^1 = [1,0]$. The transition matrices are
\[\pp^1 =
\begin{pmatrix}
  p_1 & \ 1-p_1 \\
  p_2 & \ 1-p_2
\end{pmatrix}
, \pp^0 =
\begin{pmatrix}
  q_1 & \ 1-q_1 \\
  q_2 & \ 1-q_2
\end{pmatrix},
\]
with $0 \le p_1,p_2,q_1,q_2 \le 1$. Let us first establish a sufficient condition on the four parameters $p_1$, $p_2$, $q_1$ and $q_2$ so that the RB is degenerate. For this simple model, solving the linear program \eqref{eq:relaxed_problem} amounts to finding the optimal value $0 \le \beta \le 0.5 = \alpha$ as the proportion of activation of arms in state $1$ at decision epoch $t=0$. At decision epoch $t=1$, there will then be $\beta p_1 + (0.5-\beta)q_1 + (0.5-\beta)p_2 + \beta q_2$ arms in state $1$, and the optimal value of \eqref{eq:relaxed_problem} is
\begin{align*}
  & \beta + \min \left\{ 0.5, \ \beta p_1 + (0.5-\beta)q_1 + (0.5-\beta)p_2 + \beta q_2 \right\}  \\
  = \ & \beta + \min \left\{ 0.5, \ \beta (p_1+q_2) + (0.5-\beta) (q_1+p_2) \right\}
\end{align*}
By definition, the RB is degenerate  if
\begin{equation} \label{eq:non-regular-condition}
  \argmax_{0 \le \beta \le 0.5} \left\{  \beta + \min \left\{ 0.5, \ \beta (p_1+q_2) + (0.5-\beta) (q_1+p_2) \right\} \right\} \ne 0, \ 0.5,
\end{equation}
since then $\szero(0) = \{1,2\}$. A sufficient condition for \eqref{eq:non-regular-condition} to hold is
\begin{equation} \label{eq:non-regular-condition-1}
  q_1+p_2 > 1+p_1+q_2,
\end{equation}
under which the $\argmax$ of \eqref{eq:non-regular-condition} is $\beta^* = 0.5 \times \frac{q_1+p_2 - 1}{(q_1+p_2) - (p_1+q_2)}$ and $\mm^*(1) = [0.5,0.5]$, so we activate exactly all the proportion $0.5 = \alpha$ of arms in state 1 at decision epoch $t=1$. Note that we get  $\abs{\szero(0)}=2$ and $\abs{\szero(1)} = 0$.

We next consider a stochastic model with a population of $N$ arms, where the $2$-dimensional RB satisfies \eqref{eq:non-regular-condition-1} so that it is degenerate. For any LP-compatible policy, our only choice is to activate $\beta^* N$ arms in state $1$, $(0.5-\beta^*) N$ arms in state $2$ at decision epoch $t=0$ (apply randomized rounding if necessary); and by the specific choice of values for rewards $\rr^0$, $\rr^1$, we need to activate as many arms as possible in state $1$ at decision epoch $t=1$. The expected average reward under this policy is then $\beta^* + \expect{\min \left\{ 0.5, G_N \right\}}$, where the random variables $G_N$ (indexed by $N$) inside the bracket are
\begin{equation*}
  G_N := \frac{bin(\beta^* N, p_1) + bin((0.5-\beta^*)N, q_1) + bin((0.5-\beta^*)N, p_2) + bin(\beta^* N, q_2)}{N}.
\end{equation*}
We have $\expect{G_N} = 0.5$ by definition of the value $\beta^*$. Moreover, by elementary probability theory, one has
\begin{equation*}
  \sqrt{N} \cdot \expect{0.5 - \min \left\{ 0.5, G_N \right\}} \xrightarrow[N \rightarrow \infty]{} C > 0.
\end{equation*}
Since the optimal value of \eqref{eq:relaxed_problem} is $\beta^* + 0.5$, this implies that the square root of $N$ convergence with respect to this relaxed upper-bound value can not be improved on this degenerate RB, and it is not due to the problem at decision epoch $t=0$ with $\abs{\szero(0)}>1$, but due to the fact that at $t=1$ one has $\abs{\szero(1)}=0$, and the optimal trajectory is on the boundary of two zones, namely $\left\{ \mm \in \Delta^d \mid \sum_{s \in \splus(1)} \le \alpha \right\}$ and $\left\{ \mm \in \Delta^d \mid \sum_{s \in \splus(1) } \ge \alpha \right\}$. Note that this example implies in particular that the $\calON$ convergence rate in Theorem \ref{thm:CLT-rate} is tight.

Generally speaking, for a degenerate RB, there exists some $t$ for which $\abs{\szero(t)} = 0$. This implies that $\sum_{s\in \splus(t)} m^*_s(t) = \alpha$, which means at time $t$ the optimal trajectory is on the boundary of two zones $\left\{ \mm \in \Delta^d \mid \sum_{s \in \splus(t)} \le \alpha \right\}$ and $\left\{ \mm \in \Delta^d \mid \sum_{s \in \splus(t) } \ge \alpha \right\}$. It is exactly this phenomenon that may prevent an exponentially fast convergence rate. We actually conjecture that for ``essentially all'' degenerate RBs, if $\pi$ is a LP-compatible policy, then there exists constants $\overline{C}, \underline{C} > 0$ independent of $N$ such that
\begin{equation*}
 \frac{\underline{C}}{\sqrt{N}} \le \abs{\Vpi - \rel} \le \frac{\overline{C}}{\sqrt{N}}.
\end{equation*}
This implies that as long as the RB is degenerate, we can not expect a faster than square root convergence to the LP relaxation bound, for any admissible policy.

We believe that such a result holds for ``essentially all'' models, but there are trivial models for which it will not hold. For instance, if the rewards do not depend on state nor action, then  $\Vpi=\rel$ and such a model can be degenerate.

\subsubsection{Proof of Theorem~\ref{th:rankable}} \label{subsubsec:proof4}

Assume first that the RB is rankable and let $\byopt$ be an optimal solution of the LP-problem. For each time $t$, we consider a permutation $\sigma(t)$ that orders the state by starting from the states in $\splus(t)$, then the only state in $\szero$, then the states in $\sminus(t)$ and finally the states in $\sempty$. Let $\piprio$ be the time-dependent priority policy that activates at time $t$ the states following the order $\sigma(t)$.  By \eqref{eq:piecewise_affine}, this policy is piecewise-affine (with finitely many pieces) and continuous. It is therefore Lipschitz continuous.

We now show that $\piprio$ is such that  $\piprio(\bm^*(t))=\by^*(t)$. By definition of $\splus(t)$, for all $s\in\splus(t)$, $y^*_{s,1}(t)=m^*_s(t)$. Let $s_0$ be the only state in $\szero(t)$. As $\sum_{s} y^*_{s,1}(t)=\alpha$, this implies that $\sum_{s\in\splus(t)}y^*_{s,1}(t)<\alpha$ and therefore that $y^*_{s_0,1} = \alpha - \sum_{s\in\splus(t)}y^*_{s,1}(t)$. This shows that $y^*_{s,1}(t)$ satisfies the definition of the time-varying policy \eqref{eq:piecewise_affine}. Note that if $\bm$ is such that $0\le \alpha-\sum_{s\in\splus(t)}m_{s}(t)\le m_{s_0}$, then one has:
\begin{align}
  \label{eq:local-linear}
  \piprio_s(\bm) = \left\{\begin{array}{ll}
    m_s & \text{if $s\in\splus(t)$ }\\
    \alpha - \sum_{s\in\splus(t)}m_{s}(t) & \text{if $s\in\szero(t)$ }\\
    0 & \text{otherwise}
  \end{array}\right.
\end{align}
As a byproduct (which is not used in this proof but will be used later), this also implies that $\piprio$ is locally linear if $\abs{\szero(t)}=1$ for all $t$.

Assume now that the RB is non-rankable and let $\pi$ be a time-dependent priority policy. By construction, at any time $t$, $\pi$ activates the states following a permutation $\sigma(t)$. Hence, if there exists at most one   state $s=\sigma_i(t)$ such that $\pi_{s,0}(\bm^*(t))>0$,  $\pi_{s,1}(\bm^*(t))>0$, and  for all $j<i$,  $\pi_{\sigma_j,0}(\bm^*(t))=0$, and for all $j>i$, $\pi_{\sigma_j,1}(\bm^*(t))=0$. This shows that for all $t$,  $\abs{\{s : \text{$\pi_{s,0}(\bm^*(t))>0$ and $\pi_{s,1}(\bm^*(t))>0$}\}}\le1$. Hence, $\pi$ cannot be LP-compatible because all solutions of \eqref{eq:relaxed_problem} are such that there exists a time $t$ such that $\abs{\szero(t)}\ge2$, which is implied by the assumption that the RB is non-rankable.

\subsubsection{Proof of Lemma~\ref{lem:water-filling}}  \label{subsubsec:proof5}


Fix $(\bMN, \by^*)$ as the input for the "water-filling" in dimension $d$, and let  $\bY \in \alpha \cdot \Delta^d$ be the corresponding output. Suppose that the states are sorted so that the first $s_+$ states are $\splus := \{ s^+_1, \cdots, s^+_{s_+} \}$, the next $s_0$ states are $\szero := \{ s^0_1, \cdots, s^0_{s_0} \}$, the next $s_-$ states are $\sminus := \{ s^-_1, \cdots, s^-_{s_-} \}$, and the rest $s_{\emptyset}$ states are $\sempty := \{ s^{\emptyset}_1, \cdots, s^{\emptyset}_{s_{\emptyset}} \}$. So in total $s_+ + s_0 + s_- + s_{\emptyset} = d$.

In what follows, we show how the water-filling policy can be viewed as a fixed priority policy over a larger state-space.
To see that, we define an auxiliary set of states $\shat$ with cardinal $\widehat{d} := s_+ + (2 s_0 - 1) + s_- + s_{\emptyset}$ in which we duplicate all states in $\szero(t)$ except one:
\begin{equation}
  \label{eq:Shat}
  \shat := \Big\{ s^+_1, \cdots, s^+_{s_+}, \ \underbrace{\overline{s}^0_{s_0}, \cdots, \overline{s}^0_{2}}_{\overline{\szero}}, \  s^0_{1}, \ \underbrace{\underline{s}^0_{2}, \cdots, \underline{s}^0_{s_0}}_{\underline{\szero}}, \ s^-_1, \cdots, s^-_{s_-}, \ s^{\emptyset}_1, \cdots, s^{\emptyset}_{s_{\emptyset}} \Big\},
\end{equation}
and we define the state $\widehat{\bMN}$ as:
\begin{equation} \label{eq:chain1}
  \MNshat := \begin{cases}
              \MN_s, & \mbox{if $s \in \splus \bigcup \sminus \bigcup \sempty \bigcup \{ s^0_{1} \}$}  \\
              \min(\MN_{s^0_i}, y^*_{s^0_i,1})  , & \mbox{if $s = \overline{s}^0_i \in \overline{\szero}$} \\
              \MN_{s^0_i} - \min( \MN_{s^0_i}, y^*_{s^0_i,1}), & \mbox{if $s = \underline{s}^0_i \in \underline{\szero}$}.
            \end{cases}
\end{equation}
Let $\widehat{\bY}$ be the output of a strict priority policy with the input vector $\widehat{{\bMN}}$ and where the states activated following the order as in \eqref{eq:Shat}. Let $\bY$ be defined as in
\begin{equation} \label{eq:chain3}
  Y_s := \begin{cases}
           \widehat{Y_s}, & \mbox{if $s \in \splus \bigcup \sminus \bigcup \sempty \bigcup \{ s^0_{s_0} \}$}  \\
           \widehat{Y_{\overline{s}^0_i}} + \widehat{Y_{\underline{s}^0_i}}, & \mbox{if $s = s^0_i$ with $1 \le i \le s_0-1$}.
         \end{cases}
\end{equation}
By construction, the vector $\bY$ corresponds to the vector obtained by the water-filling algorithm constructed in Section~\ref{subsec:water-filling}.



Now, consider the map chain
\begin{equation} \label{eq:chain4}
  (\bMN, \by^*) \xrightarrow[]{\eqref{eq:chain1}} (\widehat{\bMN}) \xrightarrow[]{\text{strict priority}} \widehat{\bY} \xrightarrow[]{\eqref{eq:chain3}} \bY.
\end{equation}
It should be clear that \eqref{eq:chain1} and \eqref{eq:chain3} are Lipschitz continuous functions. As a strict priority policy is Lipschitz continuous, this shows that the water-filling policy is Lipschitz continuous.

Moreover, if $\abs{\szero}\ge1$, then \eqref{eq:chain1} is locally linear (and by \eqref{eq:local-linear}, the strict priority policy used is also locally linear). As \eqref{eq:chain4} is locally linear, this implies that when the RB is non-degenerate, 
the water-filling policy constructed from this solution is therefore locally linear.



We now show by contradiction that the non-degenerate condition is necessary to obtain a locally linear policy. Assume that the problem is degenerate and consider a solution $\yopt$ of the LP problem \eqref{eq:relaxed_problem}. As the problem is degenerate, there exists $t$ such that $\szero(t)$ is empty. In the following this $t$ is fixed and omitted from the notation for simplicity.

At time $t$, we have $\sum_{s\in \splus} m^*_s = \alpha$. Let us consider an arbitrary  function from $\Delta^d$ to $\Delta^{2d}$ that is locally linear in a small neighborhood of $\bm^*$, and we shall show that the policy induced  by this function cannot be admissible. Indeed, this linear function is defined by  a matrix $\bA \in \R^{d \times d}$ so that $\byone = \bm \cdot \bA$ for any $\bm$ in this neighborhood of $\bm^*$, and in particular $\byone^* = \bm^* \cdot \bA$. Denote by $\beps \in \R^d$ a small perturbation vector so that $\bm^* + \beps \in \Delta^d$ remains in the  neighborhood. The assumption of admissibility yields
  \begin{equation} \label{eq:epsilon-vector-inequality}
    \bzero \le (\bm^*+\beps) \cdot \bA = \byone^* + \beps \cdot \bA \le \bm^* + \beps,
  \end{equation}
  where the inequalities are considered componentwise.

Consider now a state $i \in \splus$, one has $y_{i,1}^* = m^*_i$, hence \eqref{eq:epsilon-vector-inequality} implies that $(\beps \cdot \bA)_i \le \varepsilon_i$. We next replace $\beps$ by $-\beps$, note that this is possible since we are considering a neighbourhood of $\bm^*$, and we obtain the inequality in the other direction: $(\beps \cdot \bA)_i \ge \varepsilon_i$. Consequently, $(\beps \cdot \bA)_i = \varepsilon_i$ for $i \in \splus$. Similarly, for a state $i \in \sminus$, using the same idea we obtain $(\beps \cdot \bA)_i = 0$. This implies that $A_{ij}=\delta_{ij}$ for $i,j \in \splus$, and $A_{ij}=0$ for $i,j \in \sminus$. In particular, this matrix $\bA$ tells us to activate all arms in $\splus$ for any $\bm$ in a small neighbourhood of $\bm^*$. However, since $\sum_{s\in \splus} m^*_s = \alpha$, in any neighbourhood of $\bm^*$, there always exists $\bm$ such that $\sum_{s\in \splus} m_s > \alpha$. This leaves us a contradiction, since we are forced to activate strictly more than $\alpha$ arms for this $\bm$. Hence the non-degeneracy is necessary for the existence of a locally linear policy.\Halmos

\section{\bfseries\scshape{Improvements for finite values of $N$}} \label{sec:LP-update-policy}

In the previous section, we constructed a family of policies that are all asymptotically optimal as $N$ converges to infinity.  In this section, we discuss two directions that can be used to improve the performance for small values of $N$. The first one is to use the Lagrangian-optimal index of \citet{Brown2020IndexPA} -- that we call simply the LP indices. The second one is a new policy that we call the LP update policy. We will compare their performances in the numerical Section \ref{sec:numerical-experiment}.

\subsection{The LP indices}  \label{subsec:LP-indices}

The water-filling policy constructed in the previous section is asymptotically optimal regardless of the order used within the sets $\splus(t)$ and $\sminus(t)$, and it is possible to use a default priority order. This approach is for instance used \citet{zhang2021restless}, as well as in Definition~4.4 of \citet{Ve2016.6} for the infinite horizon problem. Note that as mentioned in Section 8.1 of \citet{Ve2016.6}, how to set priority ordering within $\splus$ and $\sminus$ is left open in that paper. In this section, we define the notion of LP indices, that can serve as a tie-breaking rule among $\splus$ and $\sminus$. Our later numerical experiments suggest that tie solving in $\splus$ and $\sminus$ has a clear influence on the performance of the policy and that the LP-indices perform very well.

Consider the linear program \eqref{eq:relaxed_problem}. By strong duality, there exist Lagrange multipliers  $\gamma_0^*, \dots, \gamma_{T-1}^*$ corresponding to the constraints \eqref{eq:active}, such that $\byopt$ is also an optimal solution of the following problem:
\begin{maxi!}|s|{\by \ge \mathbf{0}}{\sum_{t=0}^{T-1} \sum_{s,a} (R_s^a - a\gamma_t^*) y_{s,a} (t) }{\label{eq:prob3}}{}
  \addConstraint{y_{s,0}(t+1) + y_{s,1}(t+1)= \sum_{s',a} y_{s',a} (t)P^{a}_{s's}\qquad}{}{\forall s,t}
  \addConstraint{y_{s,0}(0) + y_{s,1}(0)=m_s(0)}{}{\forall s.}
\end{maxi!}
The above linear program \eqref{eq:prob3} can be cast into an MDP $X$ with horizon $T$, state space $\mathcal{S}$ and action space $\{0, 1\}$. The reward in state $s \in \mathcal{S}$ under action $a \in \{0, 1\}$ is $\widetilde{R_s^a} := R_s^a - a\gamma_t^*$. The transition probabilities are
$ \mathbb{P} \big( X (t+1) = y \bigm| X (t) = x, \mathrm{action} = a \big) = P^a_{x y}$. The initial condition is $X(0) \sim \mm(0)$, by interpreting $\mm(0)$
as a probability vector. The theory of stochastic dynamic programming \citet{Puterman:1994:MDP:528623} shows that there exists an optimal policy which is Markovian.

Let $Q_{s,a}(t)$ be the $Q$-values of this policy.
We define the LP-indices as
\begin{equation}\label{LP-indices}
  I_s(t) := Q_{s,1}(t) - Q_{s,0}(t).
\end{equation}
The \emph{LP-index policy} is then defined as the water-filling policy, by using the values $I_s(t)$ in \eqref{LP-indices} as a priority score to rank states within $\splus(t)$, $\sminus(t)$ and $\szero(t)$ for the water-filling procedure, at each decision epoch $t$. Note that these indices coincide with the "optimal Lagrangian index" in \citet{Brown2020IndexPA}. The LP-indices will also be defined in the infinite horizon case later in Section \ref{sec:infinite-horizon-case}.

The next result justifies the notion of LP-indices.  In particular, it implies that when the problem is rankable, the LP-indices can be used to construct directly an asymptotically optimal time-dependent priority policy by ordering the states via decreasing LP indices. Note that when the problem is non-rankable, it is really important to use the correct tie-breaking rule among the states such that $I_s(t)=0$ (for instance by using the two passes in the water-filling algorithm). Using another tie-breaking rule is in general sub-optimal, see e.g. \citet{Brown2020IndexPA}.
\begin{lem} \ \label{lem:sign-of-indices}
  The LP-indices are such that $I_s(t)\ge0$ for all $s\in\splus(t)$, $I_s(t)\le0$ for all $s\in\sminus(t)$ and $I_s(t)=0$ for all $s\in\szero(t)$.
\end{lem}

\proof{Proof. }
Let $\psi^*$ be an optimal Markovian stationary policy of \eqref{eq:prob3} formulated as a Markov decision process $X$, so that $\psi^{*}_{s,a} (t)$ is the probability of choosing action $a$ if $X(t)=s$. Our previous discussion shows that
    \begin{equation*}
      y_{s,a}^*(t) = \mathbb{P}^{\psi^*} \big( X(t)=s \big) \cdot \psi^{*}_{s,a} (t).
    \end{equation*}
Hence
\begin{itemize}
  \item $s \in \splus(t) \Rightarrow y^*_{s,0}(t)=0 \Rightarrow \psi^{*}_{s,1} (t) = 1 \mbox{ and } \psi^{*}_{s,0} (t) = 0 \Rightarrow I_s(t) > 0$;
  \item $s \in \sminus(t) \Rightarrow y^*_{s,1}(t)=0 \Rightarrow \psi^{*}_{s,1} (t) = 0 \mbox{ and } \psi^{*}_{s,0} (t) = 1 \Rightarrow I_s(t) < 0$;
  \item $s \in \szero(t) \Rightarrow 0 < y^*_{s,0}(t) < 1 \mbox{ and } 0 < y^*_{s,1}(t) < 1 \Rightarrow 0 <\psi^{*}_{s,1} (t)<1 \mbox{ and } 0<\psi^{*}_{s,0} (t)<1 \Rightarrow I_s(t) = 0$.
\end{itemize}   \Halmos

\endproof

\subsection{The LP-update policy}

One potential drawback of the Lipschitz continuous policies with their $\calON$ convergence rate proven in Theorem \ref{thm:CLT-rate} is that, the constant $C>0$ in inequality \eqref{eq:Gronwall-type-estimation} grows exponentially with the horizon $T$. Hence, for large $T$ we may need $N$ to be extremely large in order to keep $C/\sqrt{N}$ small. Intuitively, a LP-compatible policy is such that $\pi_t(\cdot)$ satisfies $\pi_t(\bm^*(t)) = \by^*(t)$. Hence, if the stochastic vector $\bMN(t)$ is close to $\bm^*(t)$, the decision vector $\bY(t) = \pi_t(\bMN(t))$ recommended by $\pi_t(\cdot)$ should be close to optimal. Yet, if $\bMN(t)$ is far from $\bm^*(t)$ (this could happen, albeit with a small probability), the decision vector recommended by $\pi_t(\cdot)$ could be far from optimal. To overcome this problem, in this section we introduce a new policy called the \emph{LP-update policy}, that recomputes a new LP-compatible policy periodically, which originates from the certainty equivalent control in dynamic programming. It works as follows:

At decision epoch $t$, we solve a relaxed LP \eqref{eq:relaxed_problem} with parameters $\{\bMN(t), T-t\}$, where the initial state is $\bMN(t)$ (as we observe at time $t$), and the time horizon is $T-t$. We choose the decision vector at time $t$ as given by this LP solution. The \emph{LP-update} policy is to apply this procedure at every decision epoch $0 \le t \le T-1$, which is summarized in Algorithm \ref{algo:LP-update}.

\begin{algorithm}[hbtp]
  \SetAlgoLined
  \SetKwInput{KwInput}{Input}
  \KwInput{Initial configuration vector $\bMN(0) = \bm(0)$ over time span  $[0,T]$.}
  \For{$t = 0,1,2,\dots,T-1$}{
    Solve LP \eqref{eq:relaxed_problem} with initial configuration vector
      $\bMN(t)$ over time span $[t,T]$. Output is $\by^*$ over time span $[t,T]$ \;
    Set $Y_{s,a}(t) := y^*_{s,a}(t)$ for $a \in \{0,1\}$ and $s \in \calS$ \;
    Obtain $\bYN(t)$ from $\bY(t)$ as in \eqref{eq:map-chain}, apply randomized rounding described in Section \ref{subsec:admissible-policy-and-randomized-rounding} if necessary \;
    Use actions $\YN_{s,a}(t)$ over all arms to advance to the next time-step\;
  }
\caption{The LP-update policy.}
  \label{algo:LP-update}
\end{algorithm}

Note that solving the LP problem~\eqref{eq:relaxed_problem} at each time steps can be quite costly. Hence, as a compromise one might do update only from time to time, and apply the water-filling policy obtained from the most recent solution of LP between two updates. For the sake of simplicity, we discuss in the following the LP-update policy that updates at every decision epoch. The following result demonstrates that the LP-update policy is asymptotically optimal with rate $\calON$, as any LP-compatible Lipschitz continuous policy does. In particular, it implies that the LP-update policy is LP-compatible, as can also be seen from Bellman's principle of optimality.


\begin{thm} \label{thm:LP-update-rate} \
  Consider a finite horizon RB. Let the LP-update policy be defined as above, and denote by $\Update$ the value of LP-update policy on a RB with parameter set $\left\{ \bm(0), T \right\}$. Then there exists a constant $C'>0$ independent of $N$ such that
  \begin{equation*}
    \abs{\Update - \rel} \le \frac{C'}{\sqrt{N}}.
  \end{equation*}
  Consequently the LP-update policy is asymptotically optimal with rate $\calON$.
\end{thm}

\proof{Proof.}
Denote by $\byoptt$ the solution of the LP \eqref{eq:relaxed_problem} with parameter set $\left\{\bMN(t),  T-t \right\}$ at decision epoch $t$. Write similarly $\bmoptt$ where $m^{t*}_s(t') = y^{t*}_{s,0}(t')+y^{t*}_{s,1}(t')$ for $t \le t' \le T-1$ and $s \in \mathcal{S}$. Bellman's principle of optimality gives
\begin{equation} \label{eq:proof-update-rel}
  \relm{\bMN(t),T-t} = \sum_{s,a}\yoptt_{s,a}(t)R^a_s + \relm{\bmoptt(t+1),T-(t+1)},
\end{equation}
and the value of the LP-update policy on parameter set $\left\{\bMN(t),  T-t \right\}$ is
\begin{align}
  \updates[T-t]{t} = & \sum_{s,a}\yoptt_{s,a}(t)R^a_s + \expect{\updates[T-(t+1)]{t+1}}. \label{eq:proof-update}
\end{align}

Denote by $Z(t):=\updates[T-t]{t} - \relm{\bMN(t),T-t}$ the difference between \eqref{eq:proof-update-rel} and \eqref{eq:proof-update}, one has $Z(T)=0$ and for all $t\in\{1\dots T-1\}$:
\begin{align*}
  \expect{Z(t)} &= \expect{\updates[T-(t+1)]{t+1} -  \relm{\bmoptt(t+1),T-(t+1)}}\\
  &=\expect{Z(t+1)} + \expect{\relm{\bMN(t+1), T-t+1} - \relm{\bmoptt(t+1),T-(t+1)}}.
\end{align*}

From the general theory of linear programming (see for instance Section 5.6.2 of \citet{10.5555/993483}), the function $V_{\mathrm{rel}} (\ \cdot \ , t): \Delta^d \rightarrow \mathbb{R} $ is Lipschitz continuous with a constant denoted $K_t$. We have:
\begin{align*}
  \abs{\Update - \rel} &= \expect{Z(0)} \le \sum_{t=0}^{T-1} \expect{K_t\norme{\bMN(t+1) - \bmoptt(t+1)}}.
\end{align*}

By Lemma \ref{lem:Markovian-transition-analysis} we have
\begin{align*}
  \bMN(t+1) &= \phi(\bYN(t)) + \bEN(t), \\
  \bmoptt(t+1) &=\phi(\byoptt(t)).
\end{align*}
Moreover, by construction $\norme{\bYN(t) - \byoptt(t)} \le 2d/N$ where the term $2d/N$ is caused by randomized rounding and is of order $\calO(\frac{1}{N})$. Recall also that $\phi(\cdot)$ is a Lipschitz function. The dominating error hence comes from $\expect{\bEN(t) \mid \bYN(t)} \le \sqrt{d} / \sqrt{N}$, by using Lemma \ref{lem:Markovian-transition-analysis}. We therefore can bound:
\begin{equation} \label{eq:Gronwall-avoided}
  \abs{\Update - \rel} \le \frac{\sqrt{d}\sum_{t=1}^T K_t}{\sqrt{N}} + \calO(\frac{1}{N}).
\end{equation}
\Halmos
\endproof

\begin{rem}[The growth rate of the constants] \label{rem:discussion-of-the-constants} \
 Let us comment on  how the positive constants $C$ in \eqref{eq:thm:CLT1} of Theorem \ref{thm:CLT-rate}, $C_1$, $C_2$ in Theorem \ref{thm:exponential-rate} and Theorem \ref{th:water-filling}, $C'$ in Theorem \ref{thm:LP-update-rate} grow with the other parameters of the RB model, in particular, the horizon $T$. As can be seen from the proofs of these theorems, $C$, $C_1$ and $C_2^{-1}$ are all chosen to be in the order of $(\widetilde{L})^T$, where $\widetilde{L}$ is an upper bound of the Lipschitz constants of a class of Lipschitz continuous maps, and is independent of both $T$ and $N$. The exponential growth with respect to $T$ is inevitable, since we rely on a Gronwall's lemma analysis in these results. The dependence of $\widetilde{L}$ on the number of states $d$ is not clear, as it depends subtly on all the other parameters of the RB model, like the entries of the transition matrices $\pp^0$, $\pp^1$ (which are themselves of size $d \times d$), as well as $\rr^0$, $\rr^1$ and $\alpha$.

 The situation is different for the constant $C'$ because the analysis of the LP-update policy does not rely on a Gronwall's type result.   If we neglect the $\calO(1/N)$ term in \eqref{eq:Gronwall-avoided}, $C'$ can be chosen as $\sqrt{d}\sum_{t=1}^T K_t$, where $K_t$ is the Lipschitz constant of $\left\{V_{\mathrm{rel}} (\ \cdot \ , t)\right\}_{t\ge 0}$. In the very recent paper \cite{brown2022fluid}, the authors analyze a model similar to ours and obtain bounds on the Lipschitz-constants of $\left\{V_{\mathrm{rel}} (\ \cdot \ , t)\right\}_{t\ge 0}$ (see Proposition~4.2 therein). Their proof can easily be adapted to our model to show that $K_t\le T-t$. This demonstrates that $C'$ can be selected to be smaller than $T^2\sqrt{d}/2$ and implies that \eqref{eq:Gronwall-avoided} can be refined into
\begin{equation*}
  \abs{\Update - \rel} \le \frac{T^2 \sqrt{d}}{2\sqrt{N}} + \calO(\frac{1}{N}).
\end{equation*}
This exibits a polynomial growth in $T$, as opposed to the exponential growth of Theorem~\ref{thm:CLT-rate}.
 
\end{rem}

It is natural to expect that the LP-update policy performs better than its non-update counterpart. In particular, the LP-update policy should become optimal exponentially fast on non-degenerate RB models. This result, as well as a more sophisticated implementation of the update idea and its analysis, are available in a followup paper \citet{gast2022lp}. We discuss the comparison between the update and non-update approaches in  more details in our numerical experiments.

\section{\bfseries\scshape{Infinite horizon case}} \label{sec:infinite-horizon-case}

In this section we study the discrete-time Markovian \emph{infinite horizon RB} model, which is defined with the parameters $\big\{ (\pp^0, \pp^1, \rr^0, \rr^1); \alpha, N \big\}$. Since the analysis follows the same line as in the finite horizon case, we shall be brief and highlight mainly the differences. In particular, we will discuss the uniform global attractor property in Theorem \ref{thm:LP-index-policy-for-IHRB}, and compare the LP indices with the classical Whittle indices in Proposition \ref{prop:LP-indices-for-IHRB}.


\subsection{Infinite-horizon LP relaxation and non-degenerate RB}  \label{subsec:infinite-lp-relaxation-and-non-degenerate-RB}
The analogue of \eqref{eq:prob1-reformulated} in the infinite horizon case is
\begin{maxi!}|s|{\pi \in \Pi}{\lim_{T \rightarrow \infty} \frac{1}{T} \mathbb{E}_{\pi} \Big[ \sum_{t=0}^{T-1} \YN_{s,a}(t) R^{a}_{s} \Big]}{\label{eq:prob4}}{V^{(N)}_{\mathrm{opt}} (\infty) =}
 \addConstraint{\sum_{s} \YN_{s,1}(t) =
 \begin{cases}
 (\floor{\alpha N}+1)/N, & \mbox{with probability } \{\alpha N\} \\
  \floor{\alpha N}/N, & \mbox{otherwise.}
 \end{cases} \ \ \forall t \label{eq:budget-constraint}}{}
 \addConstraint{\text{Arms follow the Markovian evolution \eqref{eq:markovian_evolution}}}{}
\end{maxi!}

Here $\Pi$ is the set of Markovian stationary policies. To ease the notations, we assume that $V^{(N)}_{\mathrm{opt}} (\infty)$ does not depend on the starting state. This is true for instance if the infinite horizon RB is weakly communicating \cite{Puterman:1994:MDP:528623}. This does not play a bigger role than simplifying the notations, because the UGAP property (introduced later) and our proof show that even if $V^{(N)}_{\mathrm{opt}} (\infty)$ were to depend on the initial configuration vector, it would still converge to $V^{(N)}_{\mathrm{rel}} (\infty)$. 


We next relax the constraints in \eqref{eq:budget-constraint} into the following single constraint
\begin{equation}\label{eq:relax-constraint-infinite}
  \lim_{T \rightarrow \infty} \frac{1}{T} \sum_{t=0}^{T-1} \sum_s \mathbb{E}_{\pi} \big[ \YN_{s,1}(t) \big] = \alpha,
\end{equation}
and define variables $y_{s,a}$ for $s \in \mathcal{S}$, $a \in \{0,1\}$ as
\begin{equation*}
  y_{s,a} := \lim_{T \rightarrow \infty} \frac{1}{T} \sum_{t=0}^{T-1} \expectpi{ \YN_{s,a}(t)}.
\end{equation*}
We then obtain the following linear program as the analogue of \eqref{eq:relaxed_problem}:
\begin{maxi!}|s|{\by \ge \mathbf{0}}{ \sum_{s,a} R_s^a y_{s,a} }{\label{eq:prob-inf-equiv}}{V_{\mathrm{rel}}(\infty)=}
  \addConstraint{\sum_{s} y_{s,1} = \alpha\label{eq:active-inf-equiv}}{}{}
  \addConstraint{y_{s,0} + y_{s,1}= \sum_{s',a} y_{s',a} P^{a}_{s's}\label{eq:markov-inf-equiv}\qquad}{}{\forall s}
  \addConstraint{\sum_{s,a} y_{s,a}=1}{}.{ \label{eq:init2-inf-equiv}}
\end{maxi!}
Denote by $\by^*$ an optimal solution of \eqref{eq:prob-inf-equiv}. Similarly to the finite-horizon case, we define the following four sets, which form a partition of $\mathcal{S}$.
\begin{align*}
  \splus &:= \big\{ s \in \mathcal{S} \mid y^*_{s,1} >0 \ \text{and} \ y^*_{s,0} =0 \big\}\\
  \szero &:= \big\{ s \in \mathcal{S} \mid y^*_{s,1} >0 \ \text{and} \ y^*_{s,0} >0 \big\}\\
  \sminus &:= \big\{ s \in \mathcal{S} \mid y^*_{s,1} =0 \ \text{and} \ y^*_{s,0} >0 \big\}\\
  \sempty &:= \big\{ s \in \mathcal{S} \mid y^*_{s,1} =0 \ \text{and} \ y^*_{s,0} =0 \big\}.
\end{align*}
Compared to the sets before, these sets do not dependent on $t$. Note that the unichain assumption implies that $\sempty$ is empty.

As before, we say that an infinite RB is \emph{non-degenerate} if there exists a solution $\by^*$ of \eqref{eq:prob-inf-equiv} such that $\abs{\szero}\ge1$, and is \emph{rankable} if there exists a solution $\by^*$ with $\abs{\szero}\le 1$. Similar to Equation~\eqref{eq:total-randomization}, we prove that

\begin{prop} \label{prop:number-of-randomization-infinite} \
  For any infinite horizon RB, the optimization problem \eqref{eq:prob-inf-equiv} has an optimal solution $\by^*$ satisfying $\abs{\szero} \le 1$.
\end{prop}

The proof of this claim is similar to its finite horizon counter-part around Equation~\eqref{eq:total-randomization}, except that this time we apply Theorem 4.4 of \citet{altman-constrainedMDP}, which is the same type of result stated for constrained MDP using the expected average cost criteria. Consequently, any infinite horizon RB is rankable.

\subsection{Asymptotic optimality of LP-priority policy with exponential rate}  \label{subsec:lp-prioriy-policy-with-exponential-rate}

Following Definition 4.4 of \citet{Ve2016.6}, we define the set of LP-priorities as $\Sigma := \bigcup_{\by^*} \Sigma(\by^*)$, where $\Sigma(\by^*)$ is the set of permutations $\sigma=\sigma_1\dots\sigma_d$ of the $d$ states such that any state in $\splus$ appears before any state in $\szero$, and any state in $\szero$ appears before any state in $\sminus$. We call the corresponding policy a \emph{LP-priority policy}.

By Proposition \ref{prop:number-of-randomization-infinite}, there exists $\by^*$ such that $\abs{\szero}\le 1$. We shall choose this $\by^*$ and fix $\sigma^* \in \Sigma(\by^*)$. Denote by $V^{(N)}_{\mathrm{LP}} (\infty)$ the value of the corresponding LP-priority policy. Clearly we have $V^{(N)}_{\mathrm{LP}} (\infty) \le V^{(N)}_{\mathrm{opt}} (\infty) \le V_{\mathrm{rel}} (\infty)$. We wish to show the convergence of $V^{(N)}_{\mathrm{LP}} (\infty)$ to $V_{\mathrm{rel}} (\infty)$ as $N$ goes to infinity, and provide similar rates of convergence. However, in the infinite horizon case, an additional important assumption on the model, which does not appear in the finite horizon case, must be assumed in order for the convergence to hold, for which we discuss next.

As a LP-priority policy is a strict priority policy, one can show that the following map (the analogue of \eqref{eq:map-chain})
\begin{equation} \label{eq:map-chain-inf}
  \Psi : \bMN(t) \xrightarrow[\text{policy}]{\text{LP priority}} \bY(t) = \bYN(t) \xrightarrow[\text{Markovian transition~\eqref{eq:markovian_evolution}}]{\text{each arm follows the}} \phi(\bYN(t))
\end{equation}
is a piecewise-affine and continuous function from $\Delta^d$ to $\Delta^d$, with $d$ affine pieces (see Lemma 3.1 of \citet{gast2023exponential}). Define the $t$-th iteration of maps $\Psi_{t\ge0}(\cdot)$ as $\Psi_0(\bm) = \bm$, $\Psi_{t+1}(\bm) = \Psi \big( \Psi_t(\bm) \big)$. We assume that the dynamics of $\Psi_{t\ge0}(\cdot)$ satisfies the following property:

\textbf{(Uniform Global Attractor Property (UGAP))} The vector $\mm^* \in \Delta^d$ given by the optimal solution of \eqref{eq:prob-inf-equiv} is a uniform global attractor of $\Psi_{t\geq0} (\cdot)$, i.e. for all $\epsilon>0$, there exists $T(\epsilon)>0$ such that for all $t\ge T(\epsilon)$ and all $\mm\in\Delta^d$, one has $\norme{\Psi_t(\mm)-\mm^*}\le\epsilon$.

The configuration vector $\bm^*$ is the stationary regime under the LP-priority policy, obtained from an optimal LP solution in \eqref{eq:prob-inf-equiv}. Intuitively speaking, the difference between the stochastic trajectory of the RB with $N$ arms to the deterministic dynamics $\Psi_{t\geq0} (\cdot)$ induced by the LP-priority policy will diminish as $N$ grows. The UGAP condition ensures that any initial configuration vector will eventually converge to $\bm^*$ under the deterministic dynamics. Consequently, the stochastic $N$-armed bandit in stationary regime will also concentrate on $\bm^*$ as $N$ goes to infinity, which gives asymptotic optimality. Conversely, we have found numerical examples where the deterministic dynamics is cyclic (so that UGAP is not satisfied), and the stochastic trajectories are "attracted" to this cycle. As a result, the corresponding LP-priority policies are asymptotically sub-optimal in these cases.

The next theorem is a refinement of the asymptotic optimality result in \citet{Ve2016.6} (Proposition 4.14), proving the exponential convergence rate under the additional non-degeneracy condition on the infinite horizon RB.

\begin{thm}  \label{thm:LP-index-policy-for-IHRB} \
  Consider an infinite horizon RB which is unichain and satisfies the global attractor property. Then the LP-priority policy induced by $\sigma^*$ is asymptotically optimal. Moreover, if the RB is non-degenerate and satisfies the UGAP, then the convergence rate can be shown to be exponential: there exist two constants $C_3,C_4>0$ independent of $N$ such that
  \begin{equation*}
    \abs{V^{(N)}_{\mathrm{LP}} (\infty) - V_{\mathrm{rel}} (\infty)} \le C_3 e^{-C_4 N}
  \end{equation*}
\end{thm}

The proof of this theorem aligns with Theorem 3.2 from \citet{gast2023exponential}. Let us comment on the conditions assumed for the two theorems. The proof of Theorem 3.2 from \citet{gast2023exponential} predicates that the infinite horizon RB is indexable as an initial requirement. Our current result does not hinge upon the assumption of indexability.

The non-singularity condition outlined in \citet{gast2023exponential} parallels the non-degenerate condition discussed here. As demonstrated by Remark 3.1 in \citet{gast2023exponential}, this condition is generally indispensable to guarantee an exponential rate. Yet, in contrast to the non-degenerate condition in Theorem \ref{th:water-filling} for finite horizons, its counterpart for infinite horizons is almost always met. This is because degeneracy implies $\sum_{s \in \splus} m^*_s = \alpha$, which sets an equality constraint on the model parameters.

The UGAP assumption stands as a more stringent version of the global attractor condition—a condition frequently encountered in literature (see, for instance, \cite{WeberWeiss1990} and other works referencing their findings). UGAP further demands that $\bm^*$ maintains local stability. This, in conjunction with the piecewise-affine character of $\Psi$ and the non-degeneracy of the RB, indicates the exponential stability of $\bm^*$, as detailed in Appendix C.3 of \citet{gast2023exponential}. This stability becomes a critical factor in substantiating the exponential convergence rate. Although our assumption is reinforced to pinpoint the exponential convergence rate, we could not construct a non-degenerate example that meets the global attractor criteria and for which $\bm^*$ is \emph{not} exponentially stable. We suspect that under the current problem setting, the exponentially stable condition might be implied automatically by the global attractor condition, although we do not know how to prove it.

It is worth mentioning that there is no universal method to ascertain the global attractor property. However, it can be assessed numerically. In a practical scenario, observing the dynamical system's behavior under a multitude of initial conditions is sufficient to determine if this criterion is met. As highlighted in \citet{Ve2016.6}, formulating non-priority policies in instances without the global attractor property for infinite horizon RB remains a challenge. Nonetheless, a recent progress is presented in \citet{hong2023restless}. Here, the UGAP requirement is substituted by the so-called "Synchronization Assumption", under which a convergence rate of $\calON$ is established.

\subsection{The infinite-horizon LP indices and the Whittle indices}  \label{subsec:lp-indices-and-whittle-indices}

Similar to the LP indices discussed in Section \ref{subsec:LP-indices} for the finite horizon RB, we can also define those indices in the infinite horizon case as follows: By strong duality, there exists Lagrange multiplier $\gamma^* \in \R$ such that $\by^*$ is also an optimal solution to the following linear program:
\begin{maxi!}|s|{\by \ge \mathbf{0}}{\sum_{s,a} (R_s^a - a\gamma^*) y_{s,a} }{\label{eq:prob5}}{}
  \addConstraint{y_{s,0}+ y_{s,1} = \sum_{s',a} y_{s',a} P^{a}_{s's}\qquad}{}{\forall s}
  \addConstraint{\sum_{s,a} y_{s,a}=1}{}{}
\end{maxi!}
We again transform the problem \eqref{eq:prob5} into an MDP, with the modified rewards $\widetilde{R_s^a} := R_s^a - a \gamma^*$. The value function $V_s^*$ for state $s$ satisfies the Bellman equation
\begin{align*}
  g(\gamma^*) + V_s^* & = \max_a \big\{ \widetilde{R_s^a}+\sum_{s'} V_{s'}^* \cdot P^a_{ss'} \big\} \\
    & = \max \big\{ R_s^0 + \sum_{s'} V_{s'}^* \cdot P^0_{ss'}, \ R_s^1 - \gamma^* + \sum_{s'}V_{s'}^* \cdot P^1_{ss'} \big\} \\
    & = \max \left\{ Q_s^0, \ Q_s^1 \right\},
\end{align*}
where $g(\gamma^*)$ is the optimal value of the linear program \eqref{eq:prob5}. The LP indices for the infinite horizon RB is then defined as $I_s := Q_s^1 - Q_s^0$ for state $s$. The \emph{LP-index policy} is the strict priority policy by using the values $I_s$ as a priority order to rank states within $\splus$, $\sminus$ and $\szero$ at each decision epoch.

We next recall the classical definition of Whittle indices and the concept of indexability for an infinite horizon RB (see for instance \citet{WeberWeiss1990} and \citet{article} for a general discussion on this topic). For each value $\gamma \in \R$, the value function $V_s(\gamma)$ for state $s$ satisfies a similar Bellman equation
\begin{equation} \label{eq:bellman-equation}
  g(\gamma) + V_s(\gamma) = \max_a \big\{ R_s^a - a\gamma +\sum_{s'} V_{s'}(\gamma) \cdot P^a_{ss'} \big\}.
\end{equation}
Define
\begin{equation*}
  \mathcal{S}(\gamma) := \left\{ s \in \mathcal{S} \bigg| R_s^1 - \gamma +\sum_{s'} V_{s'}(\gamma) \cdot P^1_{ss'} > R_s^0 +\sum_{s'} V_{s'}(\gamma) \cdot P^0_{ss'} \right\}.
\end{equation*}
In other words, $\mathcal{S}(\gamma)$ is the set of states for which the $\argmax$ in \eqref{eq:bellman-equation} is $a=1$. The infinite horizon RB is \emph{indexable} if $\mathcal{S}(\gamma)$ expands monotonically from $\emptyset$ to the full set $\mathcal{S}$ when $\gamma$ is decreased from $+\infty$ to $-\infty$. The Whittle index $\gamma_s$ for state $s$ is defined to be the supremum value of $\gamma$ for which $s$ belongs to $\mathcal{S}(\gamma)$: $\gamma_s := \sup \left\{ \gamma \in \R \mid s \in \mathcal{S}(\gamma) \right\}$. WIP is the strict priority policy by using the values $\gamma_s$ as a priority score to rank states within $\splus$, $\sminus$ and $\szero$ at each decision epoch. The next result shows that both the LP-index policy and WIP are LP-priority policies.

\begin{prop} \label{prop:LP-indices-for-IHRB} \
  Assume that the infinite horizon RB is unichain, so that $\sempty = \emptyset$. Then
  \begin{enumerate}
    \item $s \in \splus \Rightarrow I_s > 0$; $s \in \sminus \Rightarrow I_s < 0 $; $s \in \szero \Rightarrow I_s=0 $.
    \item If we assume furthermore that the infinite horizon RB is indexable in Whittle's sense, then their Whittle indices $\gamma_s$ satisfy: $s \in \splus \Rightarrow \gamma_s > \gamma^* $; $s \in \sminus \Rightarrow \gamma_s < \gamma^*$; $ s \in \szero  \Rightarrow \gamma_s=\gamma^*$.
  \end{enumerate}

\end{prop}

\proof{Proof.} \
  \begin{enumerate}

   \item The proof of this claim is analogue to Lemma~\ref{lem:sign-of-indices}.

   \item We first show that for any state $s \in \szero$ (if there are any), its Whittle index $\gamma_s$ is exactly $\gamma^*$, the Lagrange multiplier in \eqref{eq:prob5}. Indeed, by definition of indexability, for any $\gamma>\gamma_s$, one has $s \notin \mathcal{S}(\gamma)$; and for any $\gamma<\gamma_s$, $s \in \mathcal{S}(\gamma)$. So $\gamma_s$ is the unique value of $\gamma$ that satisfies the equality
        \begin{equation*}
          R_s^1 - \gamma +\sum_{s'} V_{s'}(\gamma) \cdot P^1_{ss'} = R_s^0 +\sum_{s'} V_{s'}(\gamma) \cdot P^0_{ss'}.
        \end{equation*}
       On the other hand, by item 2 of Proposition \ref{prop:LP-indices-for-IHRB}, the states in $\szero$ are the states with null LP index, so the above equality are satisfied with $\gamma=\gamma^*$. Consequently the Whittle index $\gamma_s$ for $s \in \szero$ is $\gamma^*$. The other two implications then follow similarly.    \Halmos
  \end{enumerate}
\endproof

\section{\bfseries\scshape{Numerical experiments}}  \label{sec:numerical-experiment}

In this numerical part, we first demonstrate that tie-solving within $\splus$ and $\sminus$ for the Lipschitz continuous policies using water-filling is important in Section \ref{subsec:tie-solving-check}. We next show the advantage of the LP-update policy to the LP-index policy on the applicant screening problem in Section \ref{subsec:applicant-screening}, a model proposed in \citet{Brown2020IndexPA}.

\subsection{Tie-solving within $\splus$ and $\sminus$} \label{subsec:tie-solving-check}

\newcommand\relsmall{V_{\mathrm{rel}}}
\newcommand\relsmallmin{V_{\mathrm{rel-min}}}

The water-filling policy defined in Section~\ref{subsec:water-filling} is not uniquely defined as it depends on the tie-breaking rule within $\splus$ and $\sminus$. In Figure~\ref{fig2}, we compare the two tie-breaking rules:
\begin{itemize}
  \item LP-index: Give priority to the highest LP-index first, defined in Section \ref{subsec:LP-indices};
  \item Random tie-solving: Ties within $\splus$ and $\sminus$ are solved according to a random priority order that is drawn at the beginning of each simulation. The reported number for this policy is the average among $100$ priority orders.
\end{itemize}
We emphasize that these two policies are LP-compatible policies: to apply them, we first solve the LP to define $\splus$ and $\sminus$ and apply a water-filling policy. The above tie-breaking rules are only used within $\splus$ and $\sminus$.  This implies that all policies are therefore asymptotically optimal.

In each case, we compute the average \emph{score} of a policy on $100$ randomly sampled models of dimension $d=10$, horizon $T=30$ and arm population $N\in\{10\dots50\}$. To generate each model, we sample the matrices $\bP^0$ and $\bP^1$ as independent uniformly distributed probability matrices and the reward vectors as uniform between $0$ and $1$.  The score is defined as follows (for ease of notation, we omit all dependence on $(\bm(0),t)$ in this section). For a given RB, recall that $\relsmall$ is the value of the linear program \eqref{eq:relaxed_problem} and let us denote by $\relsmallmin$ the value of the same linear program but where the maximization is replaced by a minimization. The value of a policy $\pi$ is $V^N_\pi$. We define the score of the policy $\pi$ as:
\begin{equation} \label{eq:score}
  \text{score}^N_\pi = \frac{V^N_\pi - \relsmallmin}{\relsmall-\relsmallmin}.
\end{equation}
The score is a number between $0$ and $1$ (higher being better). Theorem~\ref{th:water-filling} shows that, any water-filling policy is asymptotically optimal, regardless of the tie-breaking used within $\splus$ or $\sminus$, \emph{i.e.} $\lim_{N\to\infty}\text{score}^N_\pi= 1$.

\begin{figure}[hbtp]
  \centering
  \begin{tabular}{@{}cc@{}}
    \includegraphics[width=0.49\linewidth]{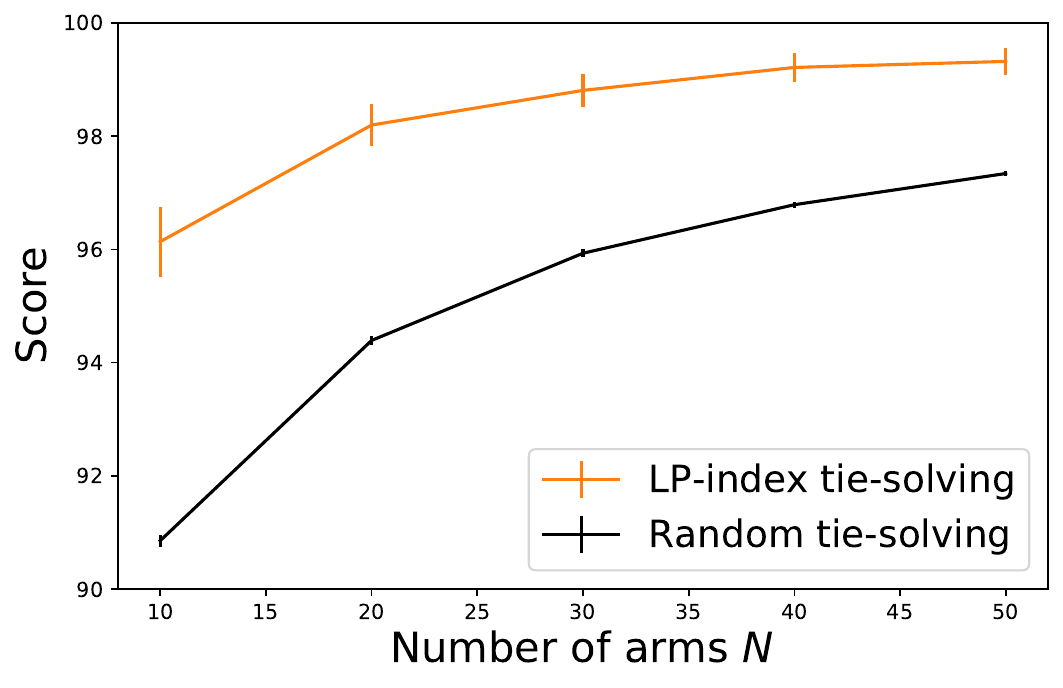}
    &\includegraphics[width=0.49\linewidth]{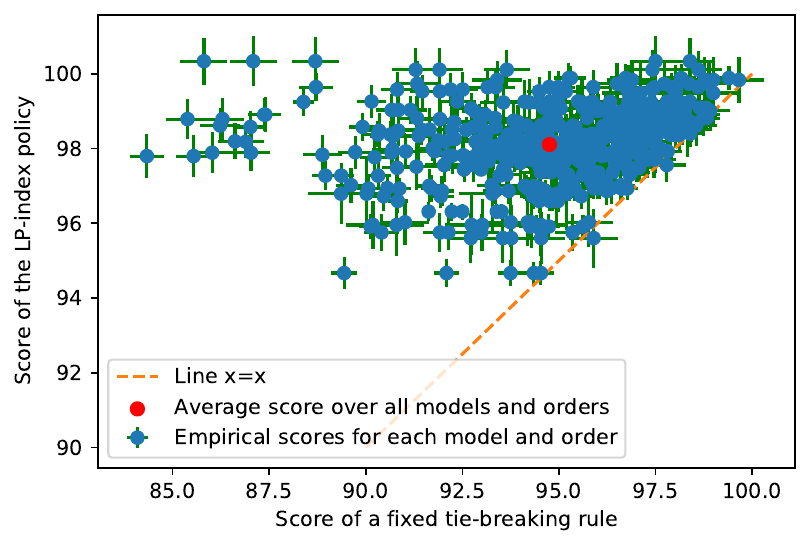}\\
    (a) Average score as a function of $N$
    &(b) LP index vs arbitrary tie-solving ($N=20$)
  \end{tabular}

  \caption{Performance of the different tie-solving among $\splus$ and $\sminus$: LP indices, and fixed priorities. We report the normalized score (in \%) as a function of the number of arms. All policies are asymptotically equivalent but the LP-index policy performs better for all finite values of $N$.}
  \label{fig2}
\end{figure}

Figure~\ref{fig2} shows that the choice of tie-solving within $\splus$ and $\sminus$ has a significant influence on the performance of the policies. On the left figure, we plot the average score over $100$ models for the LP-index policy and for $5$ random orders. This figure shows that, on average, the LP-index performs much better than a random tie-solving.  In the right figure, we fix  $N=20$ and for the same $100$ models and $5$ tie-solving rules, we plot the average score of the LP index as a function of the average score of each of the fixed tie-solving rules (this makes 500 points in total). This figure shows that the LP-index is almost always the best tie-solving rules: More precisely, among the $500$ pairs of scores considered, we observe only three points that suggest that the LP-index tie-breaking rule could be beaten, and in each case the gain of this fixed order policy is much smaller than the confidence interval, represented by the crossbars on each dot in the plots; the rest 497 pairs stay above the line $y = x$.

\subsection{Case study: applicant screening problem} \label{subsec:applicant-screening}

We discuss in this section the applicant screening problem proposed in Section 6.2.2 of \citet{Brown2020IndexPA}, and show that the LP-update policy outperforms the LP-index policy on this problem. Consider a group of $N$ applicants applying for a job. The decision maker's goal is to hire the best possible $\beta N$ applicants. Each applicant $n$ has an unknown quality level $p_n \in [0,1]$. At each decision epoch $t$, the decision maker interviews $\alpha N$ applicants and receives, for each interviewed candidate, a signal $d_n(t)\in\{0,1\}$ that is distributed according to a Bernoulli distribution of parameter $p_n$. All variables $d_n(t)$ are supposed to be independent (given $p_n$).

This problem can be seen as a RB with $N$ arms by considering a Bayesian model in which we assume that each $p_n$ is random and distributed uniformly between $0$ and $1$. Each applicant (arm)  is modeled by an MDP. The state $s_n$ of this applicant is $s_n=(a_n,b_n)$ and indicates that the posterior distribution of $p_n$ given previous observation is a beta distribution of parameters $(a_n,b_b)$: at time $0$, $a_n=b_n=1$. Afterwards, $s_n$ are updated using Bayes' rule to $(a_n+d_n,b_n+1-d_n)$ when interviewed. An applicant's state does not change when not interviewed. The rewards are set to zero during the first $T-1$ interview periods. In the final period $T$, the decision maker admits $\beta N$ applicants. The reward for admitting the applicant $n$ is $p_n$. Note that if $p_n$ is uniformly distributed, then $\expect{p_n \mid s_n} = a_n/(a_n+b_n)$. The reward for those not admitted is zero.

In our numerical study, we choose the same parameters as those used in Figure 4 of \citet{Brown2020IndexPA}, where $\alpha=\beta=0.25$, $T=5$. We compute the LP-policies by assuming that the initial state of all applicants is $(1,1)$ and consider two cases:
\begin{itemize}
  \item \textbf{Correct prior} -- In the left-panel of Figure~\ref{fig4}, the $p_n$ are generated uniformly between $0$ and $1$.
  \item \textbf{Wrong prior} -- On the right-panel of Figure~\ref{fig4}, the $p_n$ are generated using a distribution $beta(3,1)$, while the selection algorithm is constructed from a LP-relaxation that assumes that $p_n$ is uniformly distributed on $[0,1]$.
\end{itemize}
The first case fits into the framework of our paper, and in particular implies the asymptotic optimality. The second case does not fall into our framework because the transition matrices that we use to construct the policies are not the correct ones.  This second case corresponds to a decision maker having a wrong prior about the candidates.

\begin{figure}[hbtp]
  \centering
    \includegraphics[width=1\linewidth]{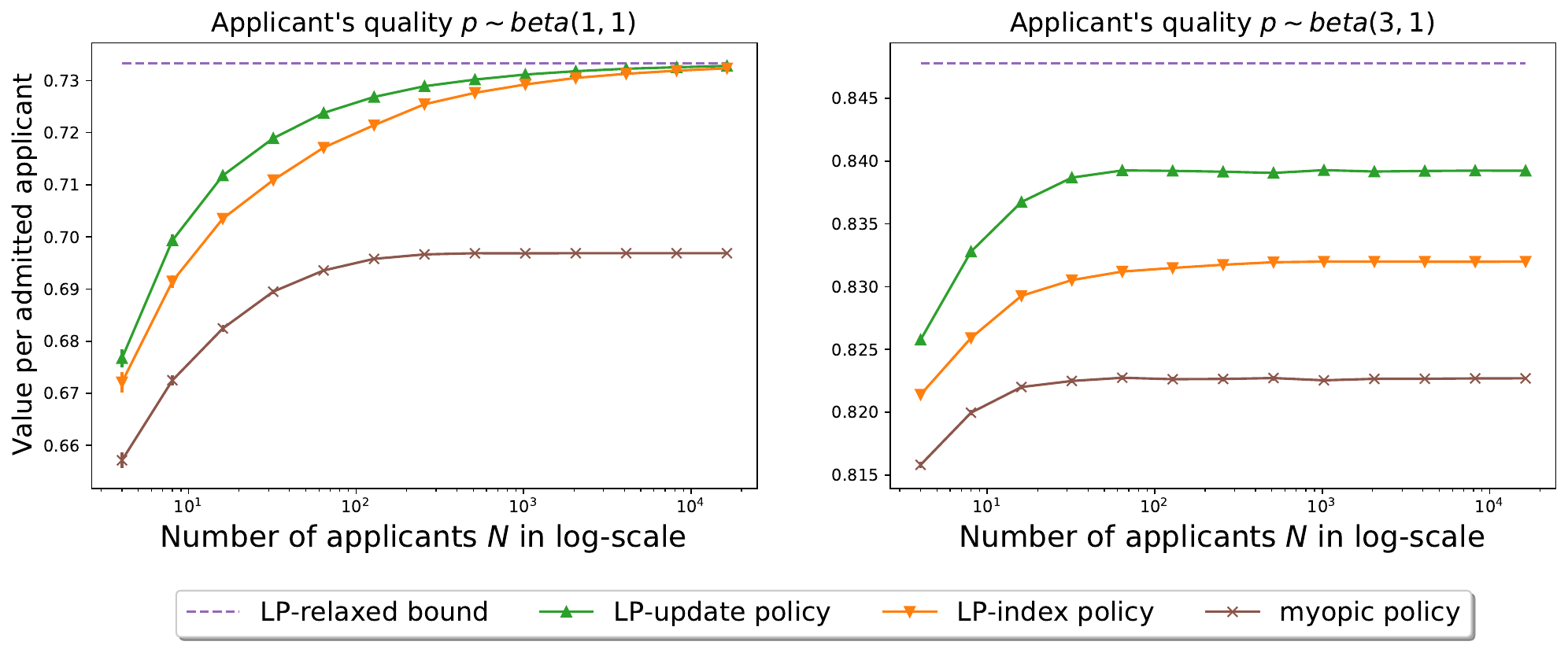}
    \caption{Performance on applicant screening problem when the decision maker knows the prior distribution of $p_n$ (left panel) or has access to a wrong prior information (right panel).}
  \label{fig4}
\end{figure}

As expected, the LP-index policy performance displayed in the left panel reproduces that of the Lagrange policy with optimal tie-breaking shown in Figure~4 of \citet{Brown2020IndexPA}, and the LP-update policy performance reproduces that of the Lagrange policy with re-optimization and optimal tie-breaking shown in Figure 6 of the same paper. For this scenario, Theorem \ref{th:water-filling} and Theorem \ref{thm:LP-update-rate} can be applied, and both the LP-index and the LP-update policies converge to the LP-relaxed bound. Moreover, the LP-update policy always outperforms the LP-index policy, with an advantage that is more apparent for $N$ in the middle range. This shows the benefit of applying updates, even in this ideal scenario.

The situation is quite different when the prior of the decision maker is wrong (right panel of Figure~\ref{fig4}). In this case, the LP-update and the LP-index policies converge to different values, that are both below the LP-relaxed bound. This is reasonable since the assumption on the $p$'s is wrong. Here, the LP-update policy outperforms the LP-index policy by a large margin, especially when $N$ is large. This is because by applying updates in this situation helps to correct the error due to the wrong assumption on the initial $p$ value of each applicant. This is yet another advantage of the LP-update policy. We expect such an advantage to hold more generally on any Bayesian RB model.

%

\section{\bfseries\scshape{Conclusion and future direction}}  \label{sec:conclusion-and-future-researches}

In this paper we propose a general framework to study LP-based policies for RB. We show that the asymptotic behavior of these policies is closely related to properties of their corresponding deterministic maps. We also illustrate the idea of applying updates and demonstrate its advantage to previously existing LP-based policies on finite horizon problems. Actually, the same idea of applying updates can be applied more broadly on Weakly Coupled MDPs, which generalize the RB model in this paper by allowing multiple actions for each arm and multiple resource constraints, as long as we can formulate the relaxed problem as a linear program. This has been accomplished in a followup paper \citet{gast2022lp}.

\section*{Acknowledgements}

This work is supported by the ANR project REFINO (ANR-19-CE23-0015).


\bibliographystyle{informs2014}
\bibliography{reference}

\end{document}